\documentclass[review,onefignum,onetabnum]{siamart190516}

\def\E{{\rm E}\,}

\usepackage{lipsum}
\usepackage{amsfonts}
\usepackage{graphicx,epstopdf}
\usepackage{epstopdf}
\usepackage{algorithmic}
\usepackage{multirow}
\usepackage{subfig}
\usepackage{enumerate}
\usepackage{algorithm}
\usepackage{graphicx}
\usepackage{textcomp}

\ifpdf
  \DeclareGraphicsExtensions{.eps,.pdf,.png,.jpg}
\else
  \DeclareGraphicsExtensions{.eps}
\fi

\newsiamremark{remark}{Remark}
\newsiamremark{hypothesis}{Hypothesis}
\crefname{hypothesis}{Hypothesis}{Hypotheses}
\newsiamthm{claim}{Claim}
\newsiamthm{assumption}{Assumption}

\headers{Global Active Subspace}{Ruilong Yue, Giray \"{O}kten}

\title{The Global Active Subspace Method\thanks{Submitted to the editors DATE.
}}

\author{Ruilong Yue\thanks{Department of Mathematics, Florida State University, Tallahassee, FL 
  (\email{ryue@math.fsu.edu}).}
\and Giray \"{O}kten\thanks{Corresponding author. Department of Mathematics, Florida State University, Tallahassee, FL 
  (\email{okten@math.fsu.edu}).}}

\usepackage{amsopn}

\ifpdf
\hypersetup{
  pdftitle={The Global Active Subspace Method},
  pdfauthor={Ruilong Yue, Giray  \"{O}kten}
}
\fi

\begin{document}

\maketitle

\begin{abstract}
We present a new dimension reduction method called the global active subspace method. The method uses expected values of finite differences of the underlying function to identify the important directions, and builds a surrogate model using the important directions on a lower dimensional subspace. The method is a generalization of the active subspace method which uses the gradient information of the function to construct a reduced model. We develop the error analysis for the global active subspace method, and present examples to compare it with the active subspace method numerically. The results show that the global active subspace method is more accurate, efficient, and robust with respect to noise or lack of smoothness in the underlying function.
\end{abstract}



\section{Introduction}
The active subspace (AS) method is a dimension reduction method introduced in Russi \cite{russi2010uncertainty} and Constantine et al. \cite{constantine2014active}, and further developed in Constantine \cite{constantine2015active} and Constantine and Gleich \cite{constantine2014computing}. The method has been applied successfully in many problems in fields of engineering, public health, quantitative finance, and geosciences (see Constantine et al. \cite{constantine2015exploiting}, Cui et al. \cite{cui2020active}, Demo et al. \cite{demo2021supervised}, Diaz et al. \cite{diaz2018modified}, Jefferson et al. \cite{jefferson2015active}, Kubler and Scheidegger\cite{kubler2019self}, Li et al. \cite{li2019surrogate}, Liu and Owen \cite{liu2022pre}, Lukaczyk et al. \cite{lukaczyk2014active}, Zhou and Peng \cite{zhou2021active}). The method uses the gradient information of a function to find the directions along which the function changes the most. The function is then approximated by one that depends only on the few important directions. Estimating the gradient accurately is crucial in successful application of the AS method, and several researchers introduced methods for this estimation, for example, Coleman et al. \cite{coleman2019gradient}, Constantine et al. \cite{constantine2015computing}, Lewis et al. \cite{lewis2016gradient}, Ma et al. \cite{ma2020support}, Navaneeth and Chakraborty \cite{navaneeth2022surrogate}, Wycoff et al. \cite{wycoff2021sequential}, and Yan et al. \cite{yan2021accelerating}.

In this paper we will introduce a generalization of the AS method, called the global active subspace method (GAS), that does not rely on gradients. Instead the method computes the expectations of first order finite-differences of function values, and uses this information to identify the important directions along which the function changes the most. The GAS method measures change in function values in a more global way by considering finite-differences as opposed to partial derivatives, and it has theoretical connections with Sobol' sensitivity indices (Sobol' \cite{sobol2001global}), a popular measure used in global sensitivity analysis (Saltelli et al. \cite{saltelli2008global}). Although the error analysis of GAS requires some smoothness on the function, we will demonstrate the robustness of the method numerically when the underlying function is noisy or lacks the necessary smoothness. 

The paper is organized as follows. Section \ref{sec:review} presents the definition of the GAS method, discusses some computational issues, and describes how to build a surrogate model using the GAS method. Section \ref{sec:alg} presents an algorithm for the method. The error analysis of the method is established in Sections \ref{sec:error} and \ref{sec:assump}. Several numerical examples are used to compare the accuracy and efficiency of the AS and GAS methods in Section \ref{sec:numerical}. We conclude in Section \ref{sec:conc}.

\section{Global Active Subspace Method}
\label{sec:review}
In this section we develop the theoretical foundation of the global active subspace (GAS) method, discuss computational issues and construction of surrogate models based on the important directions, and present its algorithm. 

In the theory of the global active subspace, we use tools from probability theory such as expectations, however, there is no inherent randomness in the functions or variables. This is the same interpretation used by Constantine et al. \cite{constantine2014active} for the active subspace method. The expectations in the paper should be viewed as a short-hand notation for Riemann-Stieltjes integrals, and the domain of our functions is $\pmb R^d$, equipped with a finite measure (in particular, probability measure) induced by cumulative distribution functions.     


\subsection{Derivation of GAS Method}
\label{sec:deriv}

Consider the $d$-dimensional real space $\pmb R^d$, equipped with the probability measure induced by a cumulative distribution function in the form $ \pmb F(\pmb z)=F_1(z_1)\cdot \ldots \cdot F_d(z_d)$, where $F_i$ are marginal cumulative distribution functions. Let $f$ be a real-valued function on $\Omega\subseteq\pmb R^d$.
Let $D_{\pmb z}$ be an operator, acting on $f$, whose value when evaluated at $(\pmb v, \pmb z)$ is given by: 
\[
D_{\pmb z}f(\pmb v,\pmb z)=[D_{\pmb z,1}f(v_{1},\pmb z),\ldots,D_{\pmb z,d}f(v_{d},\pmb z)]^T,
\] 
where
\begin{equation}
\label{def_D}
D_{\pmb z,i}f(v_{i},\pmb z)=(f(\pmb v_{\{i\}}{:}\pmb z_{-\{i\}})-f(\pmb z))/(v_{i}-z_{i}).
\end{equation}
The notation $(\pmb v_{\{i\}}{:}\pmb z_{-\{i\}})$ means the vector whose $i$th component is $v_i$, and whose $j$th component $(j\neq i)$ is $z_j$.
The notations $\pmb v_{\{i\}}$ and $v_i$ both mean the $i$th component of $\pmb v$. We use the former notation only when we are splicing and replacing a component from a vector with another.


Define the $d\times d$ matrix $\pmb C$ by
\begin{equation}
\label{def_C}
\pmb C=\E[\E[(D_{\pmb z}f)(D_{\pmb z}f)^T|\pmb z]], 
\end{equation}
where the inner conditional expectation fixes $\pmb z$ and averages over the components of $\pmb v=[v_1,\ldots,v_d]^T$, and the outside expectation averages with respect to $\pmb z$. Both $\pmb v$ and $\pmb z$ follow the same continuous probability distribution $\pmb F$. 

To get a better understanding of the matrix $\pmb C$ let us consider a simple example, a bivariate function $f(z_1,z_2)$. The value of vector $D_{\pmb z}f(\pmb v, \pmb z)$ is then
\begin{equation}
D_{\pmb z}f(\pmb v,\pmb z)=\left[\frac{f(v_1,z_2)-f(z_1,z_2)}{v_1-z_1},\frac{f(z_1,v_2)-f(z_1,z_2)}{v_2-z_2}\right]^T,
\end{equation}
and the matrix $\pmb C$ is
\begin{equation}
\label{C_two_dim}
\pmb C=\left[
\begin{array}{cc}
    C_{11}&C_{12}\\
    C_{21}&C_{22}
\end{array}
\right],
\end{equation}
where
\begin{equation}\nonumber
\begin{split}
C_{11}&=\int_{\pmb R^3}\frac{(f(v_1,z_2)-f(z_1,z_2))^2}{(v_1-z_1)^2}dF_1(v_1)d\pmb F(\pmb z),\\
C_{22}&=\int_{\pmb R^3}\frac{(f(z_1,v_2)-f(z_1,z_2))^2}{(v_2-z_2)^2}dF_2(v_2)d\pmb F(\pmb z),\\
C_{12}=C_{21}&=\int_{\pmb R^4}\frac{(f(v_1,z_2)-f(z_1,z_2))(f(z_1,v_2)-f(z_1,z_2))}{(v_1-z_1)(v_2-z_2)}dF_1(v_1)dF_2(v_2)d\pmb F(\pmb z).
\end{split}
\end{equation}

We next discuss the necessary conditions for the existence of the integrals in $\pmb C$, for the case of bivariate $f$ - the arguments generalize to higher dimensions in a straightforward way. Let $\tilde{F}(x)=\prod_{i \in S} F_i(x_i)$ where $S$ is a list of indices $i_1,\ldots, i_s$, $s\leq d$, and each $i_k$, $k=1,\ldots,s$, belongs to the set of indices $\{1,\ldots,d\}$. Then $\tilde{F}$ is an $s$-dimensional Stieltjes measure function (Kingman and Taylor \cite{kingman}). The corresponding measure $\mu_{\tilde{F}}$ takes the following value on $s$-dimensional intervals $[a,b)=[a_1,b_1)\times\ldots\times[a_s,b_s)$:
\[
\mu_{\tilde{F}}([a,b))=\prod_{k=1,\ldots,s} (F_{i_k}(b_{i_k})-F_{i_k}(a_{i_k})).
\]
Using standard measure theory techniques $\mu_{\tilde{F}}$ can be extended to Borel sets in $\pmb R^s$, whose completion is the Lebesgue-Stieltjes (LS) measurable sets in $\pmb R^s$.

Assume $f$ has continuous first-order partial derivatives. Consider the integral $C_{11}$, and define $g(v_1,z_1,z_2)$ as 
\[
g(v_1,z_1,z_2) = \begin{cases}
\left(\frac{f(v_1,z_2)-f(z_1,z_2)}{v_1-z_1}\right)^2 \text{if } v_1 \neq z_1,\\
\left(\frac{\partial f}{\partial z_1}\right)^2 \text{if } v_1 = z_1.
\end{cases}
\]
Now consider the Riemann-Stieltjes (RS) integrals
\[
\int_{ \pmb R^3\backslash \mathcal{B}_n} g(v_1,z_1,z_2) dF_1(v_1)dF_1(z_1)dF_2(z_2)
\]
where $\mathcal{B}_n=\{(v_1,z_1,z_2) \in \pmb R^3, \text{s.t. } |v_1-z_1|<1/n \},n>0$. A necessary condition for the existence of these integrals is that $g$, which is the same as $\left(\frac{f(v_1,z_2)-f(z_1,z_2)}{v_1-z_1}\right)^2$ on $\pmb R^3\backslash \mathcal{B}_n$, is bounded and continuous over $\pmb R^3\backslash \mathcal{B}_n$ for all $n>0$. Note that $g$ is continuous if $f$ is continuous over this domain.

The integral
\[
C_{11}=\int_{ \pmb R^3} \left(\frac{f(v_1,z_2)-f(z_1,z_2)}{v_1-z_1}\right)^2 dF_1(v_1)dF_1(z_1)dF_2(z_2)
\]
is defined as the limit of RS-integrals
\[
\lim_{n\rightarrow \infty} \int_{ \pmb R^3\backslash \mathcal{B}_n} g(v_1,z_1,z_2) dF_1(v_1)dF_1(z_1)dF_2(z_2).
\]
The above limit exists since the sequence defined by
\[
I_n=\int_{ \pmb R^3\backslash \mathcal{B}_n} g(v_1,z_1,z_2) dF_1(v_1)dF_1(z_1)dF_2(z_2)
\]
is an increasing sequence bounded by $\int_{ \pmb R^3} g d\tilde{F}$, where $\tilde{F}=F_1\times F_1 \times F_2$. If $g$ is bounded on $\pmb R^3$ with $g\leq C$, and $F_i$ $(i=1,2)$ are continuous functions, then 
\[
\int_{ \mathcal{B}_n} g d\tilde{F} < C \mu_{\tilde{F}}( \mathcal{B}_n) \rightarrow 0
\]
as $n\rightarrow \infty$, since in the limit $\mathcal{B}_n$ becomes a lower dimensional subspace of $\pmb R^3$. As a consequence, $C_{11}=\int_{ \pmb R^3} g d\tilde{F}$.

%
%

Now let's the consider the integral $C_{12}$, and define $g$ as
\[
g(v_1,v_2,z_1,z_2)=\frac{(f(v_1,z_2)-f(z_1,z_2))(f(z_1,v_2)-f(z_1,z_2))}{(v_1-z_1)(v_2-z_2)},
\]
if $v_1 \neq z_1$ and $v_2 \neq z_2$. If $v_1 = z_1$, then the corresponding term $\frac{f(v_1,z_2)-f(z_1,z_2)}{v_1-z_1}$ is replaced by $\frac{\partial f}{\partial z_1}$. The case $v_2 = z_2$ is similar.

The integral 
\[
C_{12}=\int_{ \pmb R^4} g(v_1,v_2,z_1,z_2) dF_1(v_1)dF_2(v_2)dF_1(z_1)dF_2(z_2)
\]
is defined as the limit of RS-integrals
\[
\lim_{n\rightarrow \infty} \int_{ \pmb R^4\backslash \mathcal{B}_n} g(v_1,v_2,z_1,z_2) dF_1(v_1)dF_2(v_2)dF_1(z_1)dF_2(z_2),
\]
where $\mathcal{B}_n$ is the union of the sets $\{(v_1,v_2,z_1,z_2) \in \pmb R^4, \text{s.t. } |v_1-z_1|<1/n \}$ and $\{(v_1,v_2,z_1,z_2) \in \pmb R^4, \text{s.t. } |v_2-z_2|<1/n \}$. The RS-integrals above exist provided $f$ is continuous and bounded over $\pmb R^4\backslash \mathcal{B}_n$ for all $n>0$.

%
%

Let
\[
J_n=\int_{ \pmb R^4\backslash \mathcal{B}_n} |g(v_1,v_2,z_1,z_2)| dF_1(v_1)dF_2(v_2)dF_1(z_1)dF_2(z_2).
\]
Note that $J_n$ is an increasing sequence bounded by $\int_{ \pmb R^4\backslash \mathcal{B}_n} |g| d\tilde{F}$ ($\tilde{F}=F_1\times F_2 \times F_1 \times F_2$), and thus its limit exists. Observe that 
\[
0\leq |g(v_1,v_2,z_1,z_2)|-g(v_1,v_2,z_1,z_2) \leq 2 |g(v_1,v_2,z_1,z_2)|,
\]
which implies
\[
\lim_{n\rightarrow \infty} \int_{ \pmb R^4\backslash \mathcal{B}_n} ( |g(v_1,v_2,z_1,z_2)|-g(v_1,v_2,z_1,z_2))dF_1(v_1)dF_2(v_2)dF_1(z_1)dF_2(z_2)
\]
exists since the sequence is increasing and bounded by the constant $2\int_{ \pmb R^4\backslash \mathcal{B}_n} |g| d\tilde{F}$. Since the limits of integrals of $|g(v_1,v_2,z_1,z_2)|-g(v_1,v_2,z_1,z_2)$ and $|g(v_1,v_2,z_1,z_2)|$ both exist, the limit
\[
C_{12}=\lim_{n\rightarrow \infty} \int_{ \pmb R^4\backslash \mathcal{B}_n} g(v_1,v_2,z_1,z_2) dF_1(v_1)dF_2(v_2)dF_1(z_1)dF_2(z_2)
\]
exists as well. Finally, if we assume $F_i$ are continuous ($i=1,2$) in addition to boundedness of $g$, then
\[
C_{12}= \int_{ \pmb R^4} g(v_1,v_2,z_1,z_2) dF_1(v_1)dF_2(v_2)dF_1(z_1)dF_2(z_2)
\] 
since $\int_{ \mathcal{B}_n} g d\tilde{F}\rightarrow 0$ as $n\rightarrow \infty$, like in the case of $C_{11}$.

\begin{remark}
Another approach we could take was to define the matrix $\pmb C$ in such a way that the finite differences are replaced by the derivatives, just like the functions $g(\cdot)$ introduced earlier, and $C$ is defined in terms of the expectations of $g(\cdot)$, as opposed to finite differences of $f$. The quadrature rule we will use to estimate these expectations does not include nodes that lead to $v_i=z_i$, thus we do not have any numerical issues related to singularities in the finite differences. Also note that the integrals in the diagonal of $\pmb C$ resemble the non-local gradient operators discussed in Du and Tian \cite{du} (page 1540).
\end{remark}
The matrix $\pmb C$ provides a global measure of change for the function $f$, and it resembles upper Sobol' sensitivity indices (Sobol' \cite{sobol2001global}), which is a popular measure used in global sensitivity analysis. Recall that the first-order upper Sobol' sensitivity indices can be written as (Saltelli et al. \cite{saltelli2010variance})
\begin{equation}
\Bar{S}_i=\frac{1}{2\sigma^2}\int{(f(\pmb v_{\{i\}}{:}\pmb z_{-\{i\}})-f(\pmb z))^2d\pmb F(\pmb z)dF_i(v_i)},
\end{equation}
where $\sigma^2$ is the variance of $f$. If we did not divide $f(\pmb v_{\{i\}}{:}\pmb z_{-\{i\}})-f(\pmb z)$ by $v_{i}-z_{i}$ in the definition of $D_{\pmb z}f(\pmb z)$ (Eqn. \ref{def_D}), then the diagonal elements of $\frac{1}{2\sigma^2}\pmb C$ would be exactly the upper Sobol' indices $\Bar{S}_1,\ldots,\Bar{S}_d$. 

The active subspace method also constructs a matrix $\pmb C$, where the matrix is obtained from the gradient of $f$. The global active subspace is based on finite-differences, not their limits, and the accuracy and robustness of these methods in problems where the gradient does not exist or difficult to estimate accurately will be explored in the numerical results. 


Like in the active subspace method, we will use the eigenvalue decomposition of $\pmb C$ to find the important variables, or directions, of $f$. As $\pmb C$ is symmetric and positive semi-definite, there exists a real eigenvalue decomposition, and the eigenvectors are mutually orthogonal. To this end, consider the eigenvalue decomposition of $\pmb C$:
\begin{equation}
\pmb C=\pmb U\Lambda \pmb U^T,\Lambda=diag(\lambda_1,\ldots,\lambda_d), \lambda_1\geq\ldots\geq \lambda_d\geq 0.
\end{equation}
Partition the eigenvalues and eigenvectors as follows: 
\begin{equation}
\Lambda=\left[
\begin{array}{cc}
    \Lambda_1 &  \\
     & \Lambda_2
\end{array}
\right],   
\pmb U=\left[
\begin{array}{cc}
    \pmb{U}_1 &  \pmb{U}_2
\end{array}
\right],
\end{equation}
where $\Lambda_1=diag(\lambda_1,\ldots,\lambda_{d_1})$ with $d_1<d$, and $\pmb{U}_1\in\pmb R^{d\times d_1}$. Introduce new variables (rotated coordinates), $\pmb{w}_1, \pmb{w}_2$ as:
\begin{equation}
\pmb{w}_1=\pmb{U}_1^T\pmb z\in\pmb R^{d_1}, \pmb{w}_2=\pmb{U}_2^T\pmb z\in\pmb R^{d-d_1}.
\end{equation}

The vector $\pmb{w}_1$, obtained through eigenvectors corresponding to larger eigenvalues, holds the information about the important directions for the function $f$. These are the directions along which the function changes the most. We will call the subspace generated by directions in $\pmb{w}_1$, the \textit{global active subspace}. Directions in the other vector $\pmb{w}_2$ are unimportant; the function does not change as much along those directions. The following lemma links the changes in function values to the eigenvalues of $\pmb C$.

\begin{lemma}\label{EDD}
For a square-integrable, Lipchitz continuous function $f(\pmb z)$ with finite second-order partial derivatives, we have
\begin{align*}
\E[\E[(D_{\pmb{w}_1}f)^T(D_{\pmb{w}_1}f)|\pmb{w}_1]]=\lambda_{1}+\ldots+\lambda_{d_1}+C_{U_1},\\
\E[\E[(D_{\pmb{w}_2}f)^T(D_{\pmb{w}_2}f)|\pmb{w}_2]]=\lambda_{d_1+1}+\ldots+\lambda_{d}+C_{U_2},
\end{align*}
where $C_{U_1},C_{U_2}$ are two constants related to the matrix $\pmb U$ and the second order derivatives of $f$. 
\end{lemma}
\begin{proof}
We have 
\begin{equation}
    f(\pmb z)=f(\pmb{U}\pmb{U}^T\pmb z)=f(\pmb{U}_1\pmb{U}_1^T\pmb z+\pmb{U}_2\pmb{U}_2^T\pmb z)=f(\pmb{U}_1\pmb{w}_1+\pmb{U}_2\pmb{w}_2).
\end{equation}
We then derive the following equations:
\begin{equation}\label{equ:210}
\begin{split}
&D_{\pmb z}f=\nabla_{\pmb z}f(\pmb z)+R_1(\pmb{z}-\pmb{v}),\\
&D_{\pmb{w}_1}f=\pmb{U}_1^T\nabla_{\pmb{z}}f(\pmb{z})+R_1(\pmb{U}_1^T(\pmb{z}-\pmb{v})),
\end{split}
\end{equation}
where $R_1(\pmb{z}-\pmb{v})$ and $R_1(\pmb{U}_1^T(\pmb{z}-\pmb{v}))$ are two vectors obtained from the remainder terms of Taylor expansions.

To derive the first equation of Eqn. (\ref{equ:210}), we write the Taylor expansion for $f(\pmb v_{\{i\}}{:}\pmb z_{-\{i\}})$ at the point $\pmb z$ for each $i$, then approximate each entry in the vector $D_{\pmb z}f$ with the Taylor expansion, and summarize the result in the vector form. To prove the second line of Eqn. (\ref{equ:210}), we write the Taylor expansion for each entry in the vector $D_{\pmb{w}_1}f$, and then rewrite the expression using $\pmb v$ and $\pmb z$ to get the result.

Then 
\begin{align*}\label{proof1}
&\E[\E[(D_{\pmb{w}_1}f)^T(D_{\pmb{w}_1}f)|\pmb{w}_1]]=\E[\E[trace((D_{\pmb{w}_1}f)(D_{\pmb{w}_1}f)^T)|\pmb{w}_1]]\\
&= trace(\E[\E[(D_{\pmb{w}_1}f)(D_{\pmb{w}_1}f)^T|\pmb{z}]])\\
&= trace(\E[\E[\pmb{U}_1^T(D_{\pmb{z}}f)(D_{\pmb{z}}f)^T\pmb{U}_1|\pmb{z}]])+trace(\E[2\pmb{U}_1^T(D_{\pmb{z}}f)\pmb a^T+\pmb{aa}^T])\\
&= \lambda_1+\ldots+\lambda_{d_1}+C_{U_1},
\end{align*}

where 
\begin{align*}
D_{\pmb{w}_1}f&=\pmb{U}_1^TD_{\pmb{z}}f+\pmb a\\
\pmb a&=R_1(\pmb{U}_1^T(\pmb{z}-\pmb{v}))-\pmb{U}_1^TR_1(\pmb{z}-\pmb{v})\\
C_{U_1}&=trace(\E[2\pmb{U}_1^T(D_{\pmb{z}}f)\pmb a^T+\pmb{aa}^T]).
\end{align*}
The second equation of the lemma can be derived similarly, and we get 
\begin{align*}
\pmb a'&=R_1(\pmb{U}_2^T(\pmb{z}-\pmb{v}))-\pmb{U}_2^TR_1(\pmb{z}-\pmb{v})\\
C_{U_2}&=trace(\E[2\pmb{U}_2^T(D_{\pmb{z}}f)\pmb a'^T+\pmb a'\pmb a'^T]).
\end{align*}

\end{proof}

\begin{remark}\label{rmk:partition}
It is interesting to compare Lemma \ref{EDD} to a similar result for the active subspace method (Lemma 2.2 of \cite{constantine2014active}). In the active subspace method, there are no constant terms in the right-hand sides of the expectations. These constants appear as we generalize the approach from gradients to first-order divided differences. Lemma \ref{EDD} shows that when $\lambda_{d_1+1}+\ldots+\lambda_{d}+C_{U_2}$ is small, the changes for function $f$ along the directions in $\pmb{w}_2$ are also small, and thus we can use the directions in $\pmb{w}_1$ to construct a reduced dimensional approximation of $f$. In the active subspace method, the important directions can be identified from the eigenvalues directly, however, in the global active subspace method, we need information about the constants $C_{U_1}$ and $C_{U_2}$ to decide which expectation in Lemma \ref{EDD} is small. If we knew the constants are similar in magnitude and much smaller than the sum of the eigenvalues, we could continue like in the active subspace method and decide on the value of $d_1$ based on the eigenvalues only. Otherwise we have to estimate the expectations in Lemma \ref{EDD} directly, to identify the important directions. A more detailed description of this estimation and the determination of $d_1$ is given in Section \ref{sec:alg}.



\end{remark}

After the global active subspace is determined, the next task is to approximate $f$ by a function that only depends on the variables in the global active subspace, $\pmb{w}_1$. This is done in a similar way to the active subspace method, by integrating out the ``unimportant" variables $\pmb{w}_2$ using conditional expectations. 
Let $\rho$ be the probability density function that corresponds to the distribution function $\pmb F$. Let the joint probability density of $\pmb{w}_1,\pmb{w}_2$ to be $\pi(\pmb{w}_1,\pmb{w}_2):=\rho({\pmb{U}}_1\pmb{w}_1+{\pmb{U}}_2\pmb{w}_2)$, and the conditional probability density of $\pmb{w}_2$ given $\pmb{w}_1$ to be $\pi_{\pmb{w}_2|\pmb{w}_1}$, then 
\begin{equation}
\label{equ:condi}
f(\pmb z)\approx F_{gas}(\pmb z)=G(\pmb{w}_1)=\int{f({\pmb{U}}_1\pmb{w}_1+{\pmb{U}}_2\pmb{u})\pi_{\pmb{w}_2|\pmb{w}_1}}(\pmb u|\pmb{w}_1)d\pmb u.    
\end{equation}
  
Note that when $\pmb z\sim \mathcal{N}(\pmb 0,\pmb{I_d})$, we have $\pi_{\pmb{w}_2|\pmb{w}_1}=\pi_{\pmb{w}_2}$. For this convenience we prefer to use normal distribution when sampling from $\pmb z$ in the numerical results. 
 
\subsection{Computational Issues} \label{sec:comiss}
To estimate $\pmb C$, we use the Monte Carlo method:
\begin{equation}\label{gMC_C}
\hat{\pmb C}=\frac1 {M_1M_2}\sum_{i=1}^{M_1}\sum_{j=1}^{M_2}(D_{\pmb z^{(i)}}f(\pmb v^{(i,j)}, \pmb z^{(i)}))(D_{\pmb z^{(i)}}f(\pmb v^{(i,j)}, \pmb z^{(i)}))^T,    
\end{equation}
where $\pmb z^{(i)}$'s and $\pmb v^{(i,j)}$'s are sampled independently from their corresponding distributions. Then we compute the eigenvalue decomposition
\begin{equation}\label{g2p9}
\hat{\pmb C}=\hat{\pmb U}\hat{\Lambda}\hat{\pmb U}^T.    
\end{equation}

An alternative way to obtain $\hat{\pmb U}$ is by computing the singular value decomposition (SVD):
\begin{align}\label{equ:svd}
\begin{aligned}
\hat{\pmb B}=\frac1{\sqrt{M_1M_2}}[&D_{\pmb z^{(1)}}f(\pmb v^{(1,1)}, \pmb z^{(1)}),\ldots,D_{\pmb z^{(1)}}f(\pmb v^{(1,M_2)}, \pmb z^{(1)}),\\
&D_{\pmb z^{(2)}}f(\pmb v^{(2,1)}, \pmb z^{(2)}),\ldots,D_{\pmb z^{(2)}}f(\pmb v^{(2,M_2)}, \pmb z^{(2)}),\\
&\ldots\\
&D_{\pmb z^{(M_1)}}f(\pmb v^{(M_1,1)}, \pmb z^{(M_1)}),\ldots,D_{\pmb z^{(M_1)}}f(\pmb v^{(M_1,M_2)}, \pmb z^{(M_1)})]\\
=\hat{\pmb U}\sqrt{\hat{\pmb\Lambda}}\hat{\pmb V}^T.&   
\end{aligned} 
\end{align}

Note that here $\hat{\pmb B}$ is a $d\times M_1M_2$ matrix satisfying $\hat{\pmb B}\hat{\pmb B}^T=\hat{\pmb C}$. For each $\pmb z^{(i)})$, there are $M_2$ vectors shown as $M_2$ columns in $\hat{\pmb B}$. 

Using $\hat{\pmb U}$ we compute estimates for $\pmb{w}_1$, $\pmb{w}_2$, and their probability density functions: $\hat{\pmb{w}}_1=\hat{\pmb{U}}^T_1\pmb z, \hat{\pmb{w}}_2=\hat{\pmb{U}}^T_2\pmb z$, $\hat\pi(\hat{\pmb{w}}_1,\hat{\pmb{w}}_2)=\rho(\hat{\pmb{U}}_1\hat{\pmb{w}}_1+\hat{\pmb{U}}_2\hat{\pmb{w}}_2)$ and $\hat\pi_{\hat{\pmb{W}}_2|\hat{\pmb{W}}_1}$. We approximate $f(\pmb z)$ by
\begin{equation}\label{equ:estcondi}
f(\pmb z)\approx \hat F_{gas}(\pmb z)=\hat G(\hat{\pmb{w}}_1)=\int{f(\hat{\pmb{U}}_1\hat{\pmb{w}}_1+\hat{\pmb{U}}_2\pmb{u})\hat\pi_{\hat{\pmb{W}}_2|\hat{\pmb{W}}_1}}(\pmb u|\hat{\pmb{w}}_1)d\pmb u.    
\end{equation}
Monte Carlo is used to estimate the integral in Eqn. (\ref{equ:estcondi}):
\begin{equation}\label{gGhat}
\hat{F}_{mc}(\pmb{z})=\hat{G}_{mc}(\hat{\pmb{w}}_1)=\frac 1 {N_1}\sum_{i=1}^{N_1}f(\hat{\pmb{U}}_1\hat{\pmb{w}}_1+\hat{\pmb{U}}_2\pmb{\hat u}^{(i)}),   
\end{equation}
where $\pmb{\hat u}^{(i)}$'s are sampled from $\hat\pi_{\hat{\pmb{W}}_2|\hat{\pmb{W}}_1}$.

\subsection{Surrogate Model Construction}
We started with $f(\pmb z)$, discovered the global active subspace directions $\pmb{w}_1$, and approximated $f(\pmb z)$ with $\hat F_{gas}(\pmb z)=\hat G(\hat{\pmb{w}}_1)$ via Eqn. (\ref{equ:estcondi}). The next step is to replace $\hat F_{gas}(\pmb z)=\hat G(\hat{\pmb{w}}_1)$ by a surrogate model that will enable its efficient computation. Among several choices in the literature, we use polynomial chaos expansion (PCE) (Xiu and Karniadakis \cite{xiu2002wiener}) for this purpose. 

The PCE of a square-integrable variable $Y=f(\pmb w_1)$ is,
$$f(\pmb w_1)=\sum_{i=0}^{\infty}{k_i\phi_i(\pmb w_1)}.$$
Here $\{\phi_i\}_{i=0}^{\infty}$ is a family of multivariate orthonormal polynomials with respect to a weight function. Xiu and Karniadakis \cite{xiu2002wiener} introduced the Askey scheme for choosing the orthonormal polynomials. They conjectured that it is optimal to choose the family of orthonormal polynomials so that the weight function corresponds to the probability density function of $\pmb w_1$. If $\pmb w_1$ has the uniform distribution on $[0,1]^d$, we pick multivariate Legendre polynomials as the basis polynomials. If $\pmb w_1$ follows the multivariate normal distribution, the multivariate Hermite polynomials are used.

In practice, one needs to estimate the coefficients $k_i$, and compute 
\begin{align}
\label{eq_est_pce}
\hat f_p(\pmb w_1)=\sum_{i=0}^{P-1}{\hat k_i\phi_i(\pmb w_1)},
\end{align}
where $P=\binom{d+p} d$ is the number of terms in the summation, and $\hat k_i$ is the estimated value for $k_i$. A popular method in the literature to estimate $\hat k_i$ is the least squares method. Let $\pmb{\Psi}\in\pmb{R}^{N\times P}$ be the coefficient matrix with $\Psi_{ni}=\phi_i(\pmb w_1^{(n)})$, and $\pmb{y}\in\pmb{R}^{N}$ be the response vector with $y_n=f(\pmb w_1^{(n)})$. The least squares approach computes 
\begin{equation}\label{LS}
\hat{\pmb{k}}=(\pmb{\Psi}^T\pmb{\Psi})^{-1}\pmb{\Psi}^T\pmb{y},
\end{equation}
which minimizes the mean square error $||\pmb{y}-\pmb{\Psi}\pmb{k}||^2$. 

Let $\hat{F}_{sgt}$ denote the surrogate model approximation to $\hat F_{gas}(\pmb w_1)$, or more accurately, to $\hat F_{mc}$ (Eqn. (\ref{gGhat})). We will assume that $\hat{F}_{sgt}$ can be made arbitrarily close to $\hat F_{mc}$ in the following sense.

\begin{assumption}
\label{sur_assump}
Suppose that there exists $\kappa_1>0$, such that the surrogate model satisfies
\begin{equation}\label{equ:sgt}
\E[(\hat{F}_{sgt}-\hat F_{mc})^2]^{\frac12}\leq \kappa_1.
\end{equation}
\end{assumption}

In our numerical results we use PCE with the least squares method to construct $\hat{F}_{sgt}$, and in that case a bound similar to Assumption \ref{sur_assump} holds with a large probability when $N$ and $p$ are large (see Hampton and Doostan \cite{hampton2015coherence} for details). The deterministic error bound in the assumption allows us to combine several deterministic error bounds together to derive an error bound for $\E[(\hat F_{sgt}-f)^2]^{\frac12}$ in Theorem \ref{error4} of Section \ref{sec:error}.


\subsection{The GAS Algorithm}
\label{sec:alg}

Algorithm \ref{alg:gas} is a summary of steps that starts with the function $f(\pmb z)$, finds the global active subspace, computes the surrogate model via PCE, and returns $\E[f(\pmb z)]$. We show the algorithm for Gaussian distributed inputs, non-Gaussian distributed inputs can be converted to Gaussian distributed inputs by using their corresponding cumulative distribution functions, following the similar way as what Remark \ref{rem_shiftedSobol} discussed later. There is nothing special about the output, $\E[f(\pmb z)]$, except that in the numerical results we will be comparing different methods by their root mean square error when they estimate $\E[f(\pmb z)]$. A computer code for this algorithm is available at: \url{https://github.com/RuilongYue/global-active-subspace}. 

As discussed in Section \ref{sec:deriv}, we assume $\pmb z$ follows $\mathcal{N}(\pmb 0,\pmb{I_d})$ so that we can obtain $\pi(\pmb{w}_2|\pmb{w}_1)$ easily. If $\pmb z$ follows another distribution, we can apply transformation techniques to change the problem to one with normal distribution.

\begin{algorithm}[h]
\caption{GAS\_PCE Algorithm}
\label{alg:gas}
\begin{algorithmic}
\STATE{\hspace*{\algorithmicindent} \textbf{Input:} $f(\pmb z)$ with $\pmb z\sim \mathcal{N}(\pmb 0,\pmb{I_d})$, $K, M_1, M_2, N, N_1, p$}
\STATE{\hspace*{\algorithmicindent} \textbf{Output:} an estimate for $\E[f(\pmb z)]$}
\STATE{Generate a random sample of size $M_1$ from $\mathcal{N}(\pmb 0,\pmb{I_d}): \pmb z^{(1)},\ldots,\pmb z^{(M_1)}$}
\STATE{Initialize $\hat{\pmb B}$ as a $d\times M_1M_2$ empty matrix}

\FOR{$i=1,\ldots,d$} 
\STATE{Use a shifted Sobol' sequence to generate the sample $\pmb v^{(l,1)},\ldots,\pmb v^{(l,M_2)}$ from $\mathcal{N}(\pmb 0,\pmb{I_d})$ (see Remark \ref{rem_shiftedSobol} for details)}
\STATE{Update $\hat{\pmb B}$ with Eqn. (\ref{equ:svd})}
\ENDFOR

\STATE{$\hat{\pmb B}=\frac1{\sqrt{M_1M_2}}\hat{\pmb B}$} 
\STATE{Compute $\hat{\pmb B}=\hat{\pmb U}\hat{\pmb \Lambda}^{\frac1 2}\hat{\pmb V}^T$}
\STATE{Choose the appropriate $d_1$, and separate $\hat{\pmb U}=[\hat{\pmb{U}}_1,\hat{\pmb{U}}_2]$}

\FOR {$l=1,\ldots,K$}
\STATE{Generate a random sample of size $N$ from $\mathcal{N}(\pmb 0,\pmb{I_{d_1}}): \pmb w_1^{(1)},\ldots,\pmb w_1^{(N)}$}
\STATE{Compute $\hat{G}_{mc}(\pmb w_1^{(i)}), i=1,\ldots,N$, by Eqn. (\ref{gGhat})} using a sample size of $N_1$
\STATE{Construct a $p$ degree PCE with $(\pmb w_1^{(i)},\hat{G}_{mc}(\pmb w_1^{(i)})), i=1,\ldots,N$, by Eqn. (\ref{LS})}
\RETURN{$\hat{k}_{0,l}$}
\ENDFOR
\RETURN{$\frac1{K}\sum_{l=1}^{K}{\hat{k}_{0,l}}$}
\end{algorithmic}
\end{algorithm}

Let $\Gamma_i = \E[\E[(D_{\pmb w_i}f)^T(D_{\pmb w_i}f)|\pmb w_i]], i= 1,\ldots,d$, where $\pmb w_i=\pmb u_i^T \pmb z$, and $\pmb u_i$ is the $i$th column of the eigenvector matrix $\pmb U$. The value of $\Gamma_i$ reflects the change in function values as measured by the average finite differences along the direction of $\pmb w_i$. In Algorithm \ref{alg:gas}, the choice of $d_1$ will be made as follows: we will estimate $\Gamma_i$'s using Monte Carlo, and choose $d_1$ such that the gap between $\hat\Gamma_{d_1}$ and $\hat\Gamma_{d_1+1}$ is the largest. Algorithm \ref{alg:Gamma} describes how to estimate $\Gamma_i$ using Monte Carlo.


\begin{remark}
\label{rem_shiftedSobol}
In Algorithm \ref{alg:gas} a sample $\pmb v^{(l,1)},\ldots,\pmb v^{(l,M_2)}$ is generated for a given $\pmb z^{(l)}$ using shifted Sobol' sequence. Here we describe the details. Let $\Phi=(\Phi_1,\ldots,\Phi_d)$ denote the CDF of the $d$-dimensional standard multivariate normal distribution, where each $\Phi_i$ is the standard normal CDF, and $\Phi(\pmb z)=(\Phi_1(z_1),\ldots,\Phi_d(z_d))$, for $\pmb z=(z_1,\ldots,z_d)$. Let $\pmb x^{(1)},\ldots,\pmb x^{(M_2)}$ be the first $M_2$ vectors of the $d$-dimensional Sobol' sequence \cite{sobol1967distribution}. We shift the Sobol' sequence by the vector $\Phi(\pmb z^{(l)})$ to obtain the shifted Sobol' sequence $\pmb y^{(1)},\ldots,\pmb y^{(M_2)}$, where
\begin{equation}
\pmb y^{(i)} = (\pmb x^{(i)}+\Phi(\pmb z^{(l)}))\mod 1, i =1,\ldots,M_2.
\end{equation}
Here addition between two vectors is componentwise, and mod 1 means the decimal part of the sum. We then set $\pmb v^{(l,i)}=\Phi^{-1}(\pmb y^{(i)}),i=1,\ldots,M_2$ in Algorithm \ref{alg:gas}. This approach has the following advantages. First, it estimates the inner expectations of the matrix $\pmb C$ (see Eqns. (\ref{def_C}) and (\ref{gMC_C})) using the quasi-Monte Carlo method (shifting by a fixed vector preserves the property of uniform distribution mod 1), which has a better rate of convergence than the Monte Carlo method. Secondly, by using a Sobol' sequence shifted by $\pmb z^{(l)}$, we avoid obtaining denominators $\pmb v^{(l,i)}- \pmb z^{(l)}$ that are too small (by controlling the sample size $M_2$) in the finite differences of Eqn. (\ref{def_D}) while estimating their expectations. 
\end{remark}

Algorithm \ref{alg:Gamma} estimates $\Gamma_i$, $i=1,\ldots,d$ using quasi-Monte Carlo simulation. In the algorithm $\Phi$ denotes the CDF of the one-dimensional standard normal distribution. Since the conditional distribution of $v^{(k)}$ given $\hat{\pmb u}_i$ and $\pmb z^{(j)}$ is $N(-\hat{\pmb u}_i^T\pmb z^{(j)}, 1)$, we are able to use similar Shifted Sobol' sequence to calculate $v^{(k)}$'s.  

In our numerical results, we used the Sobol' sequence, and excluded the first point $v'^{(1)}=0.5$ of the sequence to avoid having $\Phi^{-1}( v'^{(1)})=0$ which implies an increment of zero.


\begin{algorithm}[h]
\caption{Estimating $\Gamma_i$'s}
\label{alg:Gamma}
\begin{algorithmic}
\STATE{\hspace*{\algorithmicindent} \textbf{Input:} $f(\pmb z)$ with $\pmb z\sim \mathcal{N}(\pmb 0,\pmb{I_d})$, $M_1, M_2$, estimated eigenvector matrix $\hat{\pmb{U}}$, whose $i$th column is $\hat{\pmb u}_i$}
\STATE{\hspace*{\algorithmicindent} \textbf{Output:} estimates for $\Gamma_i$'s}

\STATE{Let $v'^{(1)},\ldots, v'^{(M_2+1)}$ be the first $M_2+1$ numbers of the one-dimensional Sobol' sequence.}
\STATE{Generate a random sample of size $M_1$ from $\mathcal{N}(\pmb 0,\pmb{I_d}): \pmb z^{(1)},\ldots,\pmb z^{(M_1)}$.}
\STATE{Set $\hat\Gamma_i=0,i=1,\ldots,d$.}
\FOR{$i=1,\ldots,d$} 

\FOR{$j=1,\ldots,M_1$} 
\FOR{$k=1,\ldots,M_2$} 
\STATE{Set $v^{(k)}=\Phi^{-1}((v'^{(k+1)}+\Phi(\hat{\pmb u}_i^T\pmb z^{(j)}))\mod 1)-\hat{\pmb u}_i^T\pmb z^{(j)}$.}
\STATE{Accumulate the sum $\hat\Gamma_i = \hat\Gamma_i + (f(\pmb z^{(j)} +  v^{(k)}\hat{\pmb u}_i)-f(\pmb z^{(j)}))/ v^{(k)}$}
\ENDFOR
\ENDFOR
\STATE{Set $\hat\Gamma_i=\frac{1}{M_1M_2}\hat\Gamma_i$}
\ENDFOR
\RETURN{$\hat\Gamma_1,\ldots,\hat\Gamma_d$}
\end{algorithmic}
\end{algorithm}

\section{Error Analysis}
\label{sec:error}
In this section we develop the error analysis for the GAS method. We use the same framework as Constantine et al. \cite{constantine2014active}, and start with describing the various error sources.
\begin{itemize}
\item $F_{gas}-f$ (Eqn. (\ref{equ:condi})). This is the inherent approximation error for both methods, AS and GAS, where the function $f$ is replaced by another function $F_{gas}$ whose domain is a subspace. 
\item $\hat F_{gas}-f$ (Eqn. (\ref{equ:estcondi})). This error results from the approximation of matrix $\pmb C$ and $\pmb U$.
\item $\hat F_{mc}-\hat F_{gas}$ (Eqn. (\ref{gGhat})). Monte Carlo error due to approximating an integral.
\item $\hat F_{sgt}-\hat F_{mc}$ (Eqn. (\ref{equ:sgt})). This error results when the surrogate model is used to approximate $\hat F_{mc}$.
\end{itemize}

In the results that follow, the norm of a matrix, $\|\pmb C\|$, is its spectral norm, and the norm of a vector, $\|\pmb v\|$, is its 2-norm. Define the radius $R$ of a bounded set $\Omega'$ as
\begin{equation}
R=\frac{1}{2} \sup_{\pmb z,\pmb v\in\Omega'}{||\pmb z-\pmb v||}.  
\end{equation}

\begin{theorem}\label{leqR}
Assume that $\E[f(\pmb z)]=0$, then 
$$\E[f(\pmb z)^2]\leq 2R^2\E[\E[(D_{\pmb z}f)^T(D_{\pmb z}f)|\pmb z]]+\kappa_2,$$
where $\Omega'$ is a bounded subset of $\Omega$ with $\int_{\Omega'}d\pmb F(\pmb z)=1-\epsilon$, and 
\begin{equation}
\kappa_2=\frac{d(2\epsilon-\epsilon^2)}{2}\sup_{\pmb z,\pmb v\in\Omega}(f(\pmb z)-f(\pmb v))^2.    
\end{equation}
\end{theorem}
\begin{proof}
We partition $\Omega\times\Omega$ into two subsets: $A=\{(\pmb z,\pmb v)|\pmb z,\pmb v\in\Omega'\}$ and $A^c$ (complement of $A$). Note that $\int_{A^c}d\pmb F(\pmb z)d\pmb F(\pmb v)=1-(1-\epsilon)^2$.

By properties of upper Sobol' indices and definition of $D_{\pmb z}f$,
\begin{align}
\E[f(\pmb z)^2]=\sigma^2&\leq\sigma^2\sum_{i=1}^d{\bar S_i}
=\frac1 2\sum_{i=1}^d\int{(f(\pmb z)-f(\pmb v_{\{i\}}{:}\pmb z_{-\{i\}}))^2d\pmb F(\pmb z)\pmb F(\pmb v)}\\ \nonumber
&\leq 2R^2\sum_{i=1}^d\int_A{((f(\pmb z)-f(\pmb v_{\{i\}}{:}\pmb z_{-\{i\}}))/(z_i-v_i))^2d\pmb F(\pmb v)d\pmb F(\pmb z)}+\kappa_2\\ \nonumber
&\leq 2R^2\E[\E[(D_{\pmb z}f)^T(D_{\pmb z}f)|\pmb z]]+\kappa_2,
\end{align}
which proves the theorem.
\end{proof}

The following four theorems establish error bounds for the sources of error discussed above. 
\begin{theorem}\label{error1}
The mean squared error of $F_{gas}$ satisfies
\[
\E[(F_{gas}-f)^2]\leq 2R^2(\lambda_{d_1+1}+\ldots+\lambda_{d}+C_{U_2})+\kappa_2.
\]
\end{theorem}
\begin{proof}
By definition of $F_{gas}$, we have $E[(F_{gas}-f)|\pmb{w}_1]=0$.
\begin{equation}
    \begin{split}
    \E[(F_{gas}-f)^2] &= \E[\E[(F_{gas}-f)^2|\pmb{w}_1]]\\
    &\leq \E[2R^2\E[(D_{\pmb{w}_2}f)^T(D_{\pmb{w}_2}f)|\pmb{w}_1]+\kappa_2]\hspace{30pt}\text{(Theorem \ref{leqR})}\\
    &= 2R^2\E[(D_{\pmb{w}_2}f)^T(D_{\pmb{w}_2}f)]+\kappa_2\\
    &= 2R^2(\lambda_{d_1+1}+\ldots+\lambda_{d}+C_{U_2})+\kappa_2.\hspace{30pt}\text{(Lemma \ref{EDD})}
    \end{split}
\end{equation}
\end{proof}

\begin{remark}
\label{eps_assump}
Note that $||\pmb{U}_2^T\hat{\pmb{U}}_2||\leq||\pmb{U}_2^T||||\hat{\pmb{U}}_2||=1$, a fact that will be used in the proof of the next theorem. We will assume, in the same theorem, 
\begin{equation}
||\pmb{U}_1\pmb{U}_1^T-\hat{\pmb{U}}_1\hat{\pmb{U}}_1^T||=||\pmb{U}_1^T\hat{\pmb{U}}_2||\leq\epsilon,
\end{equation}
for some small positive $\epsilon$. We will justify this assumption in Section \ref{sec:assump}. 
\end{remark}

\begin{theorem}\label{error2}
Let $\hat{\pmb U}=[\hat{\pmb{U}}_1\quad\hat{\pmb{U}}_2]$ and suppose $\hat{\pmb{U}}_1$ and $\hat{\pmb{U}}_2$ satisfy $||\pmb{U}_1^T\hat{\pmb{U}}_2||\leq\epsilon$. Then the mean squared error of $\hat F_{gas}$ satisfies
\[
\E[(\hat F_{gas}-f)^2]\leq 2R^2(\epsilon(\lambda_{1}+\ldots+\lambda_{d_1})^{\frac1 2}+(\lambda_{d_1+1}+\ldots+\lambda_{d})^{\frac1 2})^2+2R^2C_{U,D}+\kappa_2.
\]
\end{theorem}
\begin{proof}
From the Taylor expansions we obtain  
$$D_{\pmb{w}_1}f=\pmb{U}_1^T\nabla_{\pmb{z}}f(\pmb{z})+\pmb{a}_1,$$
$$D_{\pmb{w}_2}f=\pmb{U}_2^T\nabla_{\pmb{z}}f(\pmb{z})+\pmb{a}_2,$$
$$D_{\hat{\pmb{w}}_2}f=\hat{\pmb{U}}_2^T\nabla_{\pmb{z}}f(\pmb{z})+\hat{\pmb{a}_2},$$
where
$$\pmb{a}_1=R_1(\pmb{U}_1^T(\pmb{z}-\pmb{v}))-\pmb{U}_1^TR_1(\pmb{z}-\pmb{v}),$$
$$\pmb{a}_2=R_1(\pmb{U}_2^T(\pmb{z}-\pmb{v}))-\pmb{U}_2^TR_1(\pmb{z}-\pmb{v}),$$
$$\hat{\pmb{a}_2}=R_1(\hat{\pmb{U}}_2^T(\pmb{z}-\pmb{v}))-\hat{\pmb{U}}_2^TR_1(\pmb{z}-\pmb{v}).$$

With $\pmb{U}_1\pmb{U}_1^T+\pmb{U}_2\pmb{U}_2^T=\pmb{I_d}$, we have the following equation:
\begin{equation}
\begin{split}
D_{\hat{\pmb{w}}_2}f&=\hat{\pmb{U}}_2^T\nabla_{\pmb{z}}f(\pmb{z})+\hat{\pmb{a}}_2\\
&=\hat{\pmb{U}}_2^T(\pmb{U}_1\pmb{U}_1^T+\pmb{U}_2\pmb{U}_2^T)\nabla_{\pmb{z}}f(\pmb{z})+\hat{\pmb{a}}_2\\
&=\hat{\pmb{U}}_2^T(\pmb{U}_1D_{\pmb{w}_1}f-\pmb{U_1a_1}+\pmb{U}_2D_{\pmb{w}_2}f-\pmb{U_2a_2})+\hat{\pmb{a}}_2\\
&=\hat{\pmb{U}}_2^T\pmb{U}_1D_{\pmb{w}_1}f+\hat{\pmb{U}}_2^T\pmb{U}_2D_{\pmb{w}_2}f+\hat{\pmb{a}}_2-\hat{\pmb{U}}_2^T(\pmb U_1\pmb a_1+\pmb U_2\pmb a_2).  
\end{split}
\end{equation}
Note that when $\hat{\pmb{U}}_2={\pmb{U}}_2$, we have $\hat{\pmb{a}}_2={\pmb{a}_2}$ and $\hat{\pmb{a}}_2-\hat{\pmb{U}}_2^T(\pmb U_1\pmb a_1 +\pmb U_2\pmb a_2)=0$. 

Following similar steps as in Theorem \ref{error1}, we obtain
\begin{align*}
\E[(\hat F_{gas}-f)^2]&\leq 2R^2\E[\E[(D_{\hat{\pmb{w}}_2}f)^T(D_{\hat{\pmb{w}}_2}f)|\hat{\pmb{w}}_2]]+\kappa_2\\
&=2R^2\E[(D_{\hat{\pmb{w}}_2}f)^T(D_{\hat{\pmb{w}}_2}f)]+\kappa_2.
\end{align*}
Using the fact $||\pmb{U}_2^T\hat{\pmb{U}}_2||\leq1$, and the assumption $||\pmb{U}_1^T\hat{\pmb{U}}_2||\leq\epsilon$ (see Remark \ref{eps_assump}), we obtain
\begin{equation}\nonumber
\begin{split}
&\hspace{15pt}\E[(D_{\hat{\pmb{w}}_2}f)^T(D_{\hat{\pmb{w}}_2}f)]\\ &\leq\E[(D_{{\pmb{w}_2}}f)^T(D_{{\pmb{w}_2}}f)]+2\epsilon\E[(D_{{\pmb{w}_2}}f)^T(D_{{\pmb{w}_1}}f)]+\epsilon^2\E[(D_{{\pmb{w}_1}}f)^T(D_{{\pmb{w}_1}}f)]+C'_{U,D} \\
&\leq(\epsilon\E[(D_{{\pmb{w}_1}}f)^T(D_{{\pmb{w}_1}}f)]^{\frac12}+\E[(D_{{\pmb{w}_2}}f)^T(D_{{\pmb{w}_2}}f)]^{\frac12})^2+C'_{U,D},\text{(Cauchy–Schwarz ineq.)}\\
&\leq (\epsilon(\lambda_{1}+\ldots+\lambda_{d_1})^{\frac1 2}+(\lambda_{d_1+1}+\ldots+\lambda_{d})^{\frac1 2})^2+C_{U,D},\hspace{30pt}\text{(Lemma \ref{EDD})}
\end{split}
\end{equation}
where 
\begin{equation}\nonumber
\begin{split}
&C'_{U,D}=2\E[(\hat{\pmb{U}}_2^T\pmb{U}_1D_{\pmb{w}_1}f+\hat{\pmb{U}}_2^T\pmb{U}_2D_{\pmb{w}_2}f)^T(\hat{\pmb{a}_2}-\hat{\pmb{U}}_2^T(\pmb U_1\pmb a_1+\pmb U_2\pmb a_2))]\\
&\hspace{20pt}+\E[(\hat{\pmb{a}}_2-\hat{\pmb{U}}_2^T(\pmb U_1\pmb a_1+\pmb U_2\pmb a_2))^T(\hat{\pmb{a}_2}-\hat{\pmb{U}}_2^T(\pmb U_1\pmb a_1+\pmb U_2\pmb a_2))],\\
&C_{U,D}=C'_{U,D}+2(\epsilon(\lambda_{1}+\ldots+\lambda_{d_1})^{\frac1 2}+(\lambda_{d_1+1}+\ldots+\lambda_{d})^{\frac1 2})(\epsilon \sqrt{|C_{U_1}|}+\sqrt{|C_{U_2}}|)\\
&\hspace{20pt}+(\epsilon \sqrt{|C_{U_1}|}+\sqrt{|C_{U_2}}|)^2.
\end{split}    
\end{equation}
\end{proof}

We continue to state the final two theorems as follows. 
\begin{theorem}\label{error3}
The mean squared error of $\hat F_{mc}$ satisfies
\begin{equation}
\E[(\hat F_{mc}-f)^2]^{\frac12}\leq \left(1+\frac1 {\sqrt{N_1}}\right)\E[(\hat F_{gas}-f)^2]^{\frac12}.
\end{equation}
\end{theorem}
\begin{proof}
The conditional variance of ${f}(\pmb{z})$ given $\hat{\pmb{w}}_1$ is $\E[(\hat F_{gas}-f)^2|\hat{\pmb{w}}_1]$. Since the variance of $\hat F_{mc}$ given $\hat{\pmb{w}}_1$ is equal to the variance of ${f}(\pmb{z})$ given $\hat{\pmb{w}}_1$ divided by $N_1$, we have:
$$\E[(\hat F_{mc}-\hat F_{gas})^2]=\E[\E[(\hat F_{mc}-\hat F_{gas})^2|\hat{\pmb{w}}_1]]=\E\left[\frac{\E[(\hat F_{gas}-f)^2|\hat{\pmb{w}}_1]}{N_1}\right]=\frac{1}{N_1}\E[(\hat F_{gas}-f)^2].$$ 
Applying the triangle inequality completes the proof.
\end{proof}

\begin{theorem}\label{error4}
Let $\hat{F}_{sgt}$ be a surrogate model obtained from the sample points $(\pmb z^{(i)},\hat{G}(\pmb{z}^{(i)}))$, $i=1,\ldots,N$, where $\pmb z^{(i)}$'s are sampled at random from the distribution of $\hat{\pmb{w}}_1$ and $\hat{G}(\pmb{z}^{(i)})$ is computed by Eqn. (\ref{gGhat}). Assume that the surrogate model satisfies Assumption \ref{sur_assump}, and $||\pmb{U}_1^T\hat{\pmb{U}}_2||\leq\epsilon$ for some $\epsilon >0$. 
Then the mean squared error of $\hat{F}_{sgt}$ satisfies
\begin{align*}
\E[(\hat F_{sgt}-f)^2]^{\frac12}
&\leq 
\left(1+\frac1{\sqrt{N_1}}\right)(2R^2(\epsilon(\lambda_{1}+\ldots+\lambda_{d_1})^{\frac1 2}+(\lambda_{d_1+1}+\ldots+\lambda_{d})^{\frac1 2})^2 \\
&+2R^2C_{U,D}+\kappa_2)^{\frac12}+\kappa_1.
\end{align*}
\end{theorem}
\begin{proof}
The result follows from Theorem \ref{error2} and Theorem \ref{error3}.
\end{proof}

\section{Estimation of the Global Active Subspace}
\label{sec:assump}
In this section, we will discuss the error when $\pmb C$ and $\pmb U$ are approximated as $\hat{\pmb C}$ and $\hat{\pmb U}$. Our results are analogous to the estimates given by \cite{constantine2014computing} for the active subspace method where $\pmb C$ is obtained from the gradient information.

We start with a result from Tropp \cite{tropp2012user}.

\begin{lemma}\label{yiduiXj}
Consider the finite sequence ${\pmb X_j}$ of independent, random, Hermitian matrices with dimension $d$. Assume that $\E[\pmb X_j]=0$, and $\lambda_{\max}(\pmb X_j)\leq R$ almost surely, where $\lambda_{\max}(\pmb X_j)$ represents the largest eigenvalue of matrix $\pmb X_j$. Compute the norm of the total variance:
\begin{equation}
\sigma^2=\left\|\sum_{j}\E[\pmb X_j^2]\right\|,
\end{equation}
where $\pmb X_j^2$ is the square of the matrix $\pmb X_j$. We have for any $\tau>0$,
\begin{equation}
P\left\{\lambda_{\max}(\sum_{j}\pmb X_j)\geq\tau\right\}\leq
\begin{cases}
d\exp{(-3\tau^2/8\sigma^2)},&\tau\leq\sigma^2/R,\\
d\exp{(-3\tau/8R)},&\tau>\sigma^2/R.
\end{cases}
\end{equation}
\end{lemma}

The following lemma presents the error bound for $\|\hat{\pmb C}-\pmb C\|$.

\begin{lemma}\label{assum1}
Let $f$ be a Lipschitz continuous function on its domain $\Omega$ such that $\|D_{\pmb z}f(\pmb v, \pmb z)\|\leq L$ for any $\pmb v, \pmb z \in \Omega$, for some constant $L$. Define the matrices:
\begin{align*}
\Tilde{\pmb C}&=\frac1{M_1}\sum_{j=1}^{M_1}\E[D_{\pmb z}fD_{\pmb z}f^T|\pmb z^{(j)}],\\
\Tilde{\pmb C}_j  &= \frac1{M_2}\sum_{i=1}^{M_2}(D_{\pmb z}f(\pmb v^{(j,i)}, \pmb z^{(j)}))(D_{\pmb z}f(\pmb v^{(j,i)}, \pmb z^{(j)}))^T,\\
\pmb C_j &= \E[D_{\pmb z}fD_{\pmb z}f^T|\pmb z^{(j)}].
\end{align*}
Assume that the events $A=\{\|\Tilde{\pmb C}-{\pmb C}\|\leq\epsilon'\|\pmb C\|\}$ and $A_j=\{\|\Tilde{\pmb C}_j-{\pmb C_j}\|\leq\epsilon''\|\pmb C_j\|\}$, $j=1,\ldots M_1$, are independent. Then for $\epsilon',\epsilon''>0$, 
\begin{equation}\label{equ:lemma42}
P\left\{\|\hat{\pmb C}-\pmb C\|\leq
\frac{\epsilon''}{M_1}\sum_{j=1}^{M_1}\|\E[D_{\pmb z}fD_{\pmb z}f^T|\pmb z^{(j)}]\|+\epsilon'\|\pmb C\|\right\}\geq(1-\delta'(M_1))\prod_{j=1}^{M_1}(1-\delta_j'(M_2)),
\end{equation}
where 
$\delta'(M_1),\delta_j'(M_2)\rightarrow 0$ as $M_1,M_2\rightarrow\infty$.
\end{lemma}
\begin{proof}
We have
\begin{equation}\label{equ:1stpart}
\begin{split}
&P\{\|\Tilde{\pmb C}-\pmb C\|\geq\epsilon'\|\pmb C\|\}\\
&= P\{(\lambda_{\max}(\Tilde{\pmb C}-\pmb C)\geq\epsilon'\|\pmb C\|\hspace{2pt}) \cup (\hspace{2pt} \lambda_{\max}(\pmb C-\Tilde{\pmb C})\geq\epsilon'\|\pmb C\|)\}\\
&\leq P\{\lambda_{\max}(\Tilde{\pmb C}-\pmb C)\geq\epsilon'\|\pmb C\|\} + P\{\lambda_{\max}(\pmb C-\Tilde{\pmb C})\geq\epsilon'\|\pmb C\|\}\\
&= P\left\{\lambda_{\max}\left(\sum_{j=1}^{M_1}(\E[D_{\pmb z}fD_{\pmb z}f^T|\pmb z^{(j)}]-\pmb C)\right)\geq M_1\epsilon'\|\pmb C\|\right\}+\\
&\hspace{20pt} P\left\{\lambda_{\max}\left(\sum_{j=1}^{M_1}(\pmb C-\E[D_{\pmb z}fD_{\pmb z}f^T|\pmb z^{(j)}])\right)\geq M_1\epsilon'\|\pmb C\|\right\}.
\end{split}
\end{equation}

Note that
\begin{equation}
\E[\E[D_{\pmb z}fD_{\pmb z}f^T|\pmb z^{(j)}]-\pmb C]=\E[\pmb C-\E[D_{\pmb z}fD_{\pmb z}f^T|\pmb z^{(j)}]]=0.
\end{equation}

By the assumption, we get $\|\E[D_{\pmb z}fD_{\pmb z}f^T|\pmb z]\|\leq L^2$ and $\|\pmb C\|\leq L^2$. As $\pmb C$ is positive semidefinite, we have
\[
\lambda_{\max}(\E[D_{\pmb z}fD_{\pmb z}f^T|\pmb z^{(j)}]-\pmb C)\leq\lambda_{\max}(\E[D_{\pmb z}fD_{\pmb z}f^T|\pmb z^{(j)}])=\|\E[D_{\pmb z}fD_{\pmb z}f^T|\pmb z^{(j)}]\|\leq L^2
\]
and
\[
\lambda_{\max}(\pmb C-\E[D_{\pmb z}fD_{\pmb z}f^T|\pmb z^{(j)}])\leq\lambda_{\max}(\pmb C)=\|\pmb C\|\leq L^2.
\]
Now we can apply Lemma \ref{yiduiXj}. Note that $\sigma^2$ is $\|\sum_{j=1}^{M_1}(\E[(\E[D_{\pmb z}fD_{\pmb z}f^T|\pmb z^{(j)}]-\pmb C)^2])\|$. Substituting $R$ by $L^2$, and $\tau$ by $M_1\epsilon'\|\pmb C\|=M_1\epsilon'\lambda_{\max}(C)$, we get
\begin{equation}\nonumber
P\{\|\Tilde{\pmb C}-\pmb C\|\geq\epsilon'\|\pmb C\|\}\leq
\begin{cases}
2d\exp{(-3M_1^2\epsilon'^2\lambda_{\max}(\pmb C)^2/8\sigma^2)},\epsilon'\leq\sigma^2/(M_1\lambda_{\max}(\pmb C)L^2)\\
2d\exp{(-3M_1\epsilon'\lambda_{\max}(\pmb C)/8L^2)},\epsilon'>\sigma^2/(M_1\lambda_{\max}(\pmb C)L^2).
\end{cases}
\end{equation}
The right-hand side of the above inequality goes to zero as $M_1\rightarrow\infty$ (fixing the other parameters). Setting the right-hand side of the above inequality to $\delta'(M_1)$, we rewrite the inequality as follows.
\begin{equation}
P\{\|\Tilde{\pmb C}-\pmb C\|\geq\epsilon'\|\pmb C\|\}\leq\delta'(M_1).
\end{equation}

Next consider the term $D_{\pmb z}f(\pmb v,\pmb z^{(j)})D_{\pmb z}f(\pmb v,\pmb z^{(j)})^T-\E[D_{\pmb z}fD_{\pmb z}f^T|\pmb z^{(j)}]$. Similar to the derivation of Eqn. (\ref{equ:1stpart}), we obtain
\begin{align*}
&P\left\{\left\|\Tilde{\pmb C}_j-\pmb C_j\right\|\geq\epsilon''\|\pmb C_j\|\right\}\\
&\leq P\left\{\lambda_{\max}\left(\sum_{i=1}^{M_2}(D_{\pmb z}f(\pmb v^{(j,i)}, \pmb z^{(j)}))(D_{\pmb z}f(\pmb v^{(j,i)}, \pmb z^{(j)}))^T-\pmb C_j\right)\geq M_2\epsilon''\|\pmb C_j\|\right\}+\\
&\hspace{10pt} P\left\{\lambda_{\max}\left(\pmb C_j-\sum_{i=1}^{M_2}(D_{\pmb z}f(\pmb v^{(j,i)}, \pmb z^{(j)}))(D_{\pmb z}f(\pmb v^{(j,i)}, \pmb z^{(j)}))^T\right)\geq M_2\epsilon''\|\pmb C_j\|\right\},
\end{align*}
then use Lemma \ref{yiduiXj} to conclude
\begin{equation}\nonumber
\begin{split}
&P\left\{\left\|\Tilde{\pmb C}_j-{\pmb C}_j\right\|\geq\epsilon''\|{\pmb C}_j\|\right\}\leq\delta_j'(M_2),
\end{split}
\end{equation}
where $\delta_j'(M_2)$ goes to zero as $M_2\rightarrow\infty$. Note that the probabilities above are conditional probabilities, as we have already sampled the $\pmb z^{(j)}$'s.

Using the following inequality
\begin{equation}\nonumber
\begin{split}
&\|\hat{\pmb C}-\pmb C\|\leq\|\hat{\pmb C}-\Tilde{\pmb C}\|+\|\Tilde{\pmb C}-\pmb C\|\leq\frac1{M_1}\sum_{j=1}^{M_1}\left\|\Tilde{\pmb C}_j-{\pmb C}_j\right\|+\|\Tilde{\pmb C}-\pmb C\|,
\end{split}
\end{equation}
and the assumption of independence, we obtain 
\begin{equation}
P\left\{\|\hat{\pmb C}-\pmb C\|\leq
\frac{\epsilon''}{M_1}\sum_{j=1}^{M_1}\|{\pmb C}_j\|+\epsilon'\|\pmb C\|\right\}\geq P(A,A_1,\ldots,A_{M_1}) = P(A)\prod_{j=1}^{M_1}P(A_j),
\end{equation}
which proves the theorem since $P\{A_j\}\geq 1-\delta_j'(M_2)$ and $P\{A\}\geq1-\delta'(M_1)$.
\end{proof}

\begin{remark}
\label{remark_limitone}
Note that $(1-\delta'(M_1))\prod_{j=1}^{M_1}(1-\delta_j'(M_2))\rightarrow1$ as $M_1, M_2 \rightarrow\infty$. To see that, observe for a fixed $\epsilon''$, $\delta_j'(M_2)\leq k_1e^{-k_2M_2}$, where $k_1,k_2$ are constants. If $M_1=M_2$, $\prod_{j=1}^{M_1}(1-\delta_j'(M_2))\leq(1-k_1e^{-k_2M_2})^{M_2}$, and the upper bound converges to $1$ as $M_2\rightarrow\infty$.
\end{remark}

Next we discuss the assumption $||\pmb{U}_1^T\hat{\pmb{U}}_2||\leq\epsilon$ mentioned in Remark \ref{eps_assump}.
\begin{theorem}\label{epsilonGAS}
Let $f$ be a Lipschitz continuous function on its domain $\Omega$ such that $\|D_{\pmb z}f(\pmb v, \pmb z)\|\leq L$ for any $\pmb v, \pmb z \in \Omega$, for some constant $L$. Let $\lambda_1\geq \lambda_2 \geq \ldots \geq\lambda_d$ be the eigenvalues of matrix
\[
\pmb C=\E[\E[(D_{\pmb z}f)(D_{\pmb z}f)^T|\pmb z]],
\] 
and let $\epsilon',\epsilon''>0$, such that $$\frac{\epsilon''}{M_1}\sum_{j=1}^{M_1}\|\E[D_{\pmb z}fD_{\pmb z}f^T|\pmb z^{(j)}]\|+\epsilon'\lambda_1\leq\frac{\lambda_{d_1}-\lambda_{d_1+1}}5.$$ 
Then the probability in Eqn. (\ref{equ:lemma42}) can be made arbitrarily close to one, by taking $M_1,M_2$ sufficiently large (see Remark \ref{remark_limitone}), and therefore with high probability, we have
\begin{equation}\label{eq:leqGAS}
||\pmb{U}_1^T\hat{\pmb{U}}_2||\leq\frac{4(\frac{\epsilon''}{M_1}\sum_{j=1}^{M_1}\|\E[D_{\pmb z}fD_{\pmb z}f^T|\pmb z^{(j)}]\|+\epsilon'\lambda_1)}{\lambda_{d_1}-\lambda_{d_1+1}}.
\end{equation}
\end{theorem}
\begin{proof}
The result follows from Lemma \ref{assum1} and Corollary 8.1.11 of Golub and Van Loan \cite{golub2013matrix}.
\end{proof}

Theorem \ref{epsilonGAS} shows that for sufficiently large sample sizes the assumption in Theorem \ref{error2} is satisfied with high probability, which completes the error analysis for the GAS method. 

\begin{remark}
There is a similar result to Theorem \ref{epsilonGAS} for the active subspace method given by Theorem 3.13 of Constantine \cite{constantine2015active}. Theorem 3.13 proves that if we choose the constants $M_1$ (the sample size when sampling their matrix $\pmb C$) sufficiently large and $h$ sufficient small, then with high probability, 
\begin{equation}\label{eq:leqAS}
||\pmb{U}_1^T\hat{\pmb{U}}_2||\leq\frac{4\sqrt d\gamma_h(\sqrt d\gamma_h+2L)}{(1-\epsilon')\lambda_{d_1}-(1+\epsilon')\lambda_{d_1+1}}+\frac{4\epsilon'\lambda_1}{\lambda_{d_1}-\lambda_{d_1+1}},
\end{equation}
where $\lambda_1\geq \lambda_2 \geq \ldots \geq\lambda_d$ are the eigenvalues of the matrix $\pmb C=\E[(\nabla_{\pmb z}f)(\nabla_{\pmb z}f)^T]$, $\epsilon'\leq\frac{\lambda_{d_1}-\lambda_{d_1+1}}{5\lambda_1}$, $L$ is an upper bound for the norm of the gradient of $f$, and we assume $||\hat\nabla f-\nabla f||\leq\sqrt{d}\gamma_h$. The constant $\gamma_h$ is related to the step size $h$ used in estimating $\nabla f$.
One observation we make comparing the upper bounds of Eqns. (\ref{eq:leqGAS}) and (\ref{eq:leqAS}) is that for the GSA method, we expect the upper bound to reach a stable limiting value as we increase the Monte Carlo sample size $M_1$, however, for the AS method, the increment $h$ has to be chosen optimally, and the estimation of the gradient in noisy functions or in limited precision may prove to be problematic in finding a small $\gamma_h$.
\end{remark}

\section{Numerical Results}
\label{sec:numerical}
\subsection{Example 1: Normal distributed noise}
Constantine and Gleich \cite{constantine2014computing} consider the function $f_0(\pmb z)=\frac12\pmb z^T\pmb{Az},z\in[0,1]^{10}$, and compute the eigenvalues of the gradient based matrix $\hat{\pmb C}$ for the active subspace method, where $\pmb A$ is a symmetric and positive definite matrix, and the domain of $f$ is equipped with the cumulative distribution function of uniform distribution on $[0,1]^{10}$. We construct the matrix by letting $\pmb A=\pmb{Q\Lambda Q}^T$, where $\pmb Q$ is a randomly generated $d\times d$ orthogonal matrix, and $\pmb\Lambda$ is a diagonal matrix. 
Here we will add a noise term to $f$, and consider 
\begin{equation}
f(\pmb z)=f_0(\pmb z)+\epsilon,\epsilon\sim \mathcal{N}(0,\sigma^2).
\end{equation}
Our objective is to compare the robustness of the active subspace (AS) and global active subspace (GAS) methods in estimating the eigenvalues ($\Gamma_i$'s) and the first two eigenvectors of their corresponding matrices (the gradient based matrix for AS and the finite-differences based matrix for GAS defined by Eqn. (\ref{def_C})), as the variance of the noise term increases as $\sigma^2=10^{-4}, 0.01$, and $1$. 

For the AS method we use a Monte Carlo sample size of $10,000$ and a finite-difference increment of $h=10^{-1}, 10^{-3}$, and $10^{-5}$, while estimating the expected values of the approximated gradients. For the GSA method, we estimate the matrix $\pmb{C}$ using three scenarios for the two Monte Carlo sample sizes: $M_1=10,000$ and $M_2=1$; $M_1=1,000$ and $M_2=10$; $M_1=100$ and $M_2=100$. In this way we use the same number of Monte Carlo samples ($10,000$) for both AS and GAS methods, to ensure a fair comparison of their estimates. When sampling $\pmb v^{(l,i)}$'s given $\pmb z^{l}$, we use the shifted Sobol' sequences discussed in Remark \ref{rem_shiftedSobol}.

Figure \ref{testas} plots the estimated eigenvalues and the first two eigenvectors together with their no-noise values for AS, for the different increments of $h$ and variances $\sigma^2$ considered. The no-noise eigenvalues and eigenvectors in Fig. \ref{testas} mean the eigenvalues and eigenvectors of $\pmb{C}$ for the case $\epsilon=0$. In this estimation we used a Monte Carlo sample size of $100,000$, and an increment of $h=10^{-6}$, for the AS method. 

Likewise, Figure \ref{testgas} plots $\hat\Gamma_i$'s and the first two eigenvectors with their no-noise values for GAS, for the different sample sizes and variances considered. When estimating $\Gamma_i$'s of no-noise function $f_0$, we used $M_1 = 100,000$ and $M_2 = 10$. When estimating $\Gamma_i$'s of $f$, we used $M_1 = 10,000$ and $M_2 = 10$ in Algorithm \ref{alg:Gamma}.

The accuracy of both methods in estimating the eigenvalues ($\Gamma_i$'s) and eigenvectors declines as the variance of the noise term increases. However, the deterioration happens much faster in the AS method compared to the GAS method. This can be seen by comparing how close the eigenvalues are to the no-noise eigenvalues as $\sigma$ increases in Figure \ref{testas}, with how close $\Gamma_i$ values are to the no-noise $\Gamma_i$ values as $\sigma$ increases in Figure \ref{testgas}. A similar comparison of how close the eigenvectors are to the no-noise eigenvectors in AS and GAS provide similar conclusions.


The numerical results show that the best increment for AS is $h=0.1$. The best choice of sample sizes $M_1,M_2$ for GAS, with the constraint $M_1M_2=10,000$, is $M_1=10,000$ and $M_2=1$. This suggests for this specific function, it is far more important to sample the domain as thoroughly as possible, and spend less sampling resources on estimating the expected values of the finite-differences at each sample point.

\begin{figure}[htpb!] 
    \centering
    \subfloat[$\sigma=0.01$]{\includegraphics[width=0.3\textwidth]{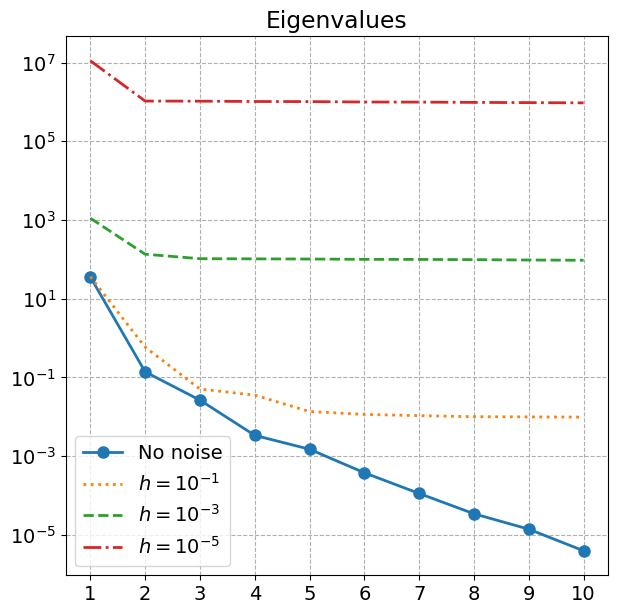}}
    \subfloat[$\sigma=0.1$]{\includegraphics[width=0.3\textwidth]{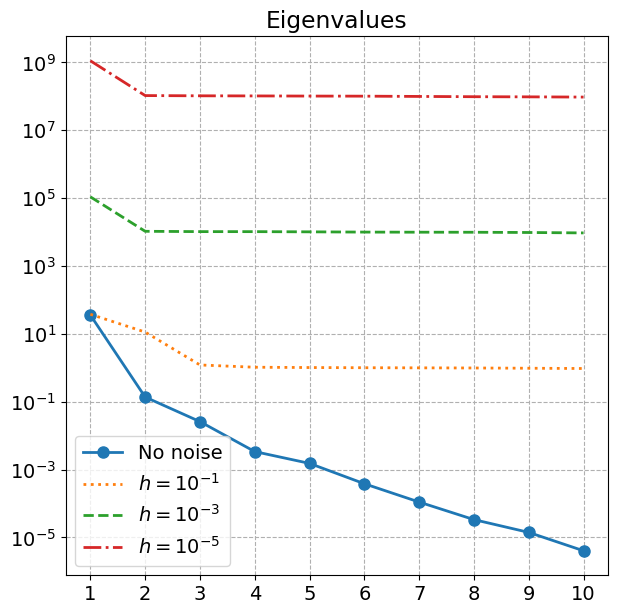}}
    \subfloat[$\sigma=1$]{\includegraphics[width=0.3\textwidth]{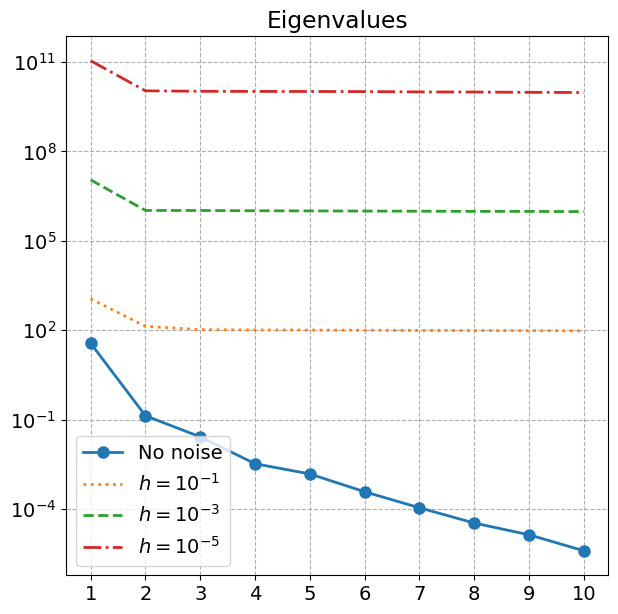}}\\
    \subfloat[$\sigma=0.01$]{\includegraphics[width=0.3\textwidth]{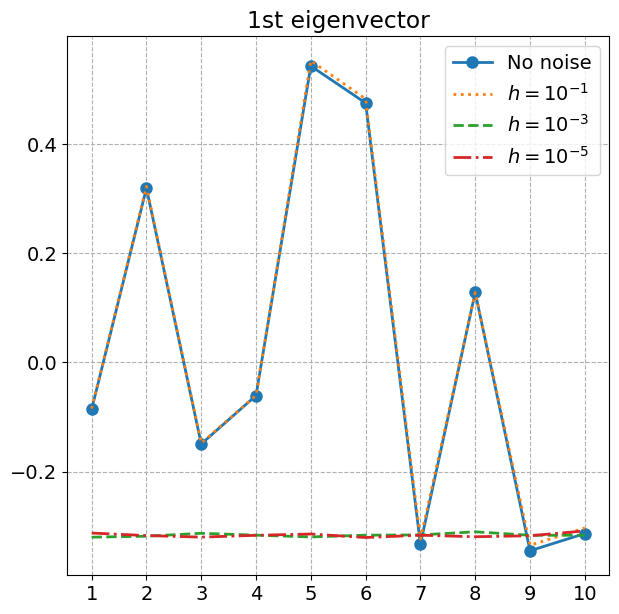}}
    \subfloat[$\sigma=0.1$]{\includegraphics[width=0.3\textwidth]{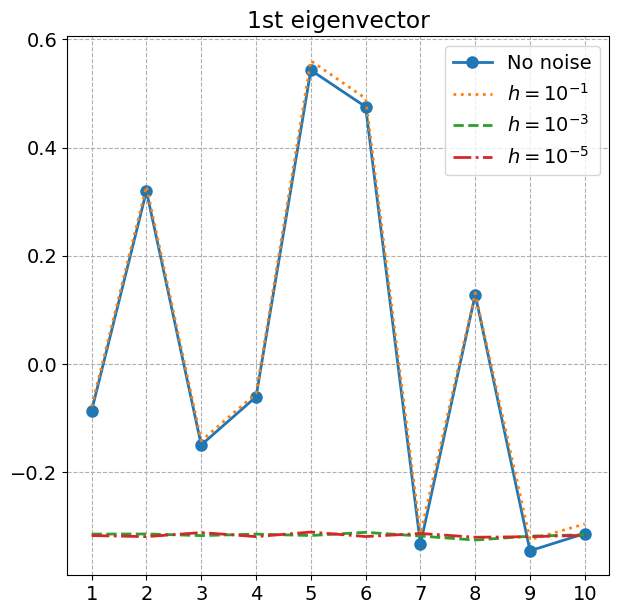}}
    \subfloat[$\sigma=1$]{\includegraphics[width=0.3\textwidth]{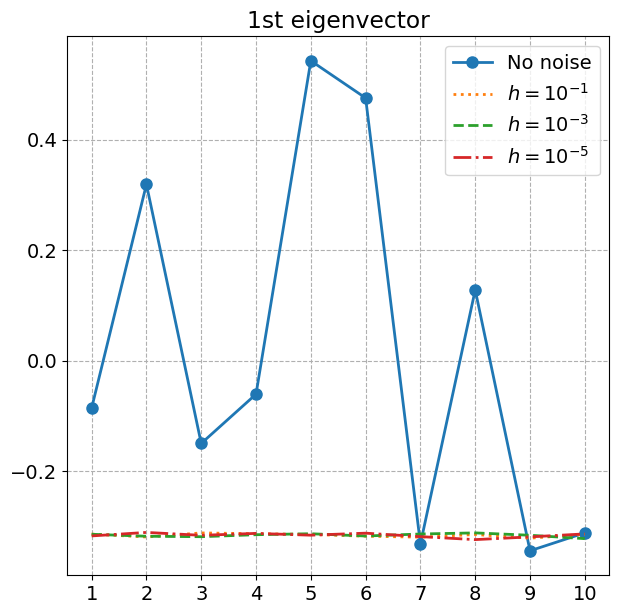}}\\
    \subfloat[$\sigma=0.01$]{\includegraphics[width=0.3\textwidth]{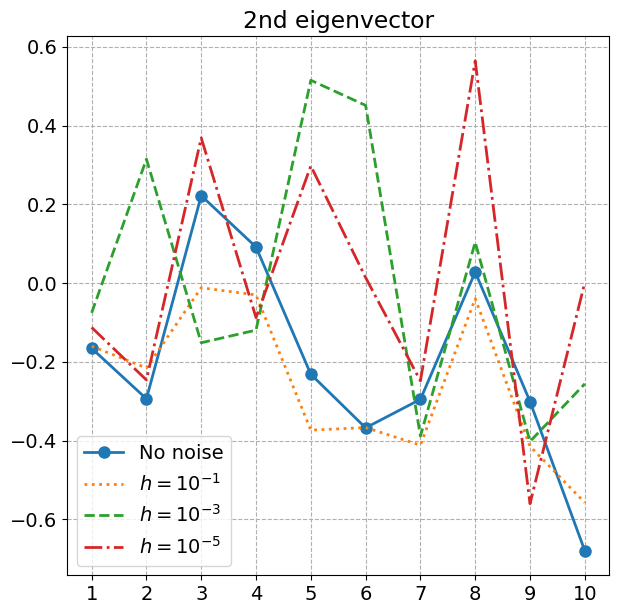}}
    \subfloat[$\sigma=0.1$]{\includegraphics[width=0.3\textwidth]{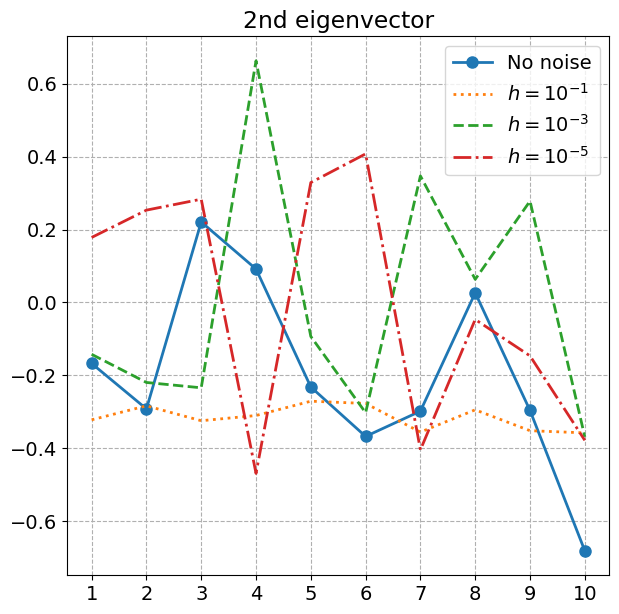}}
    \subfloat[$\sigma=1$]{\includegraphics[width=0.3\textwidth]{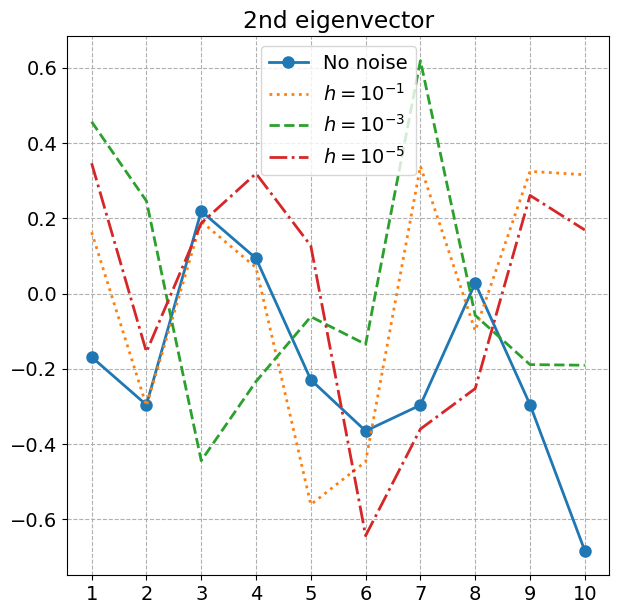}}\\
    \caption{Estimating eigenvalues and eigenvectors for Example 1: AS results.}
    \label{testas}
\end{figure}

\begin{figure}[htpb!] 
    \centering
    \subfloat[$\sigma=0.01$]{\includegraphics[width=0.3\textwidth]{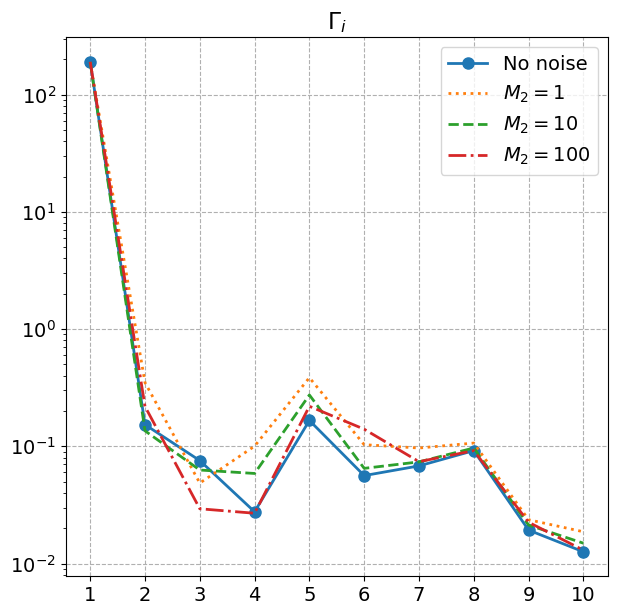}}
    \subfloat[$\sigma=0.1$]{\includegraphics[width=0.3\textwidth]{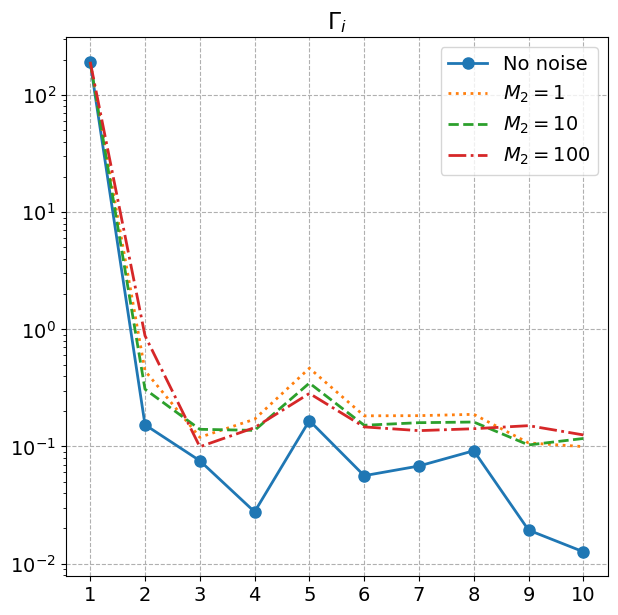}}
    \subfloat[$\sigma=1$]{\includegraphics[width=0.3\textwidth]{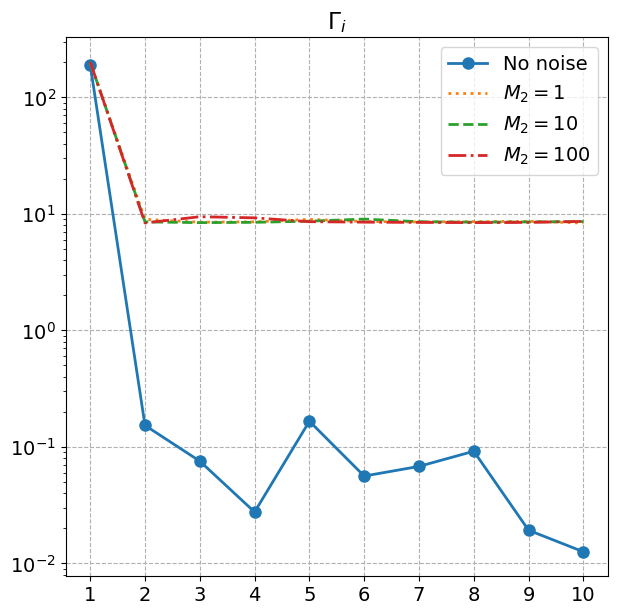}}\\
    \subfloat[$\sigma=0.01$]{\includegraphics[width=0.3\textwidth]{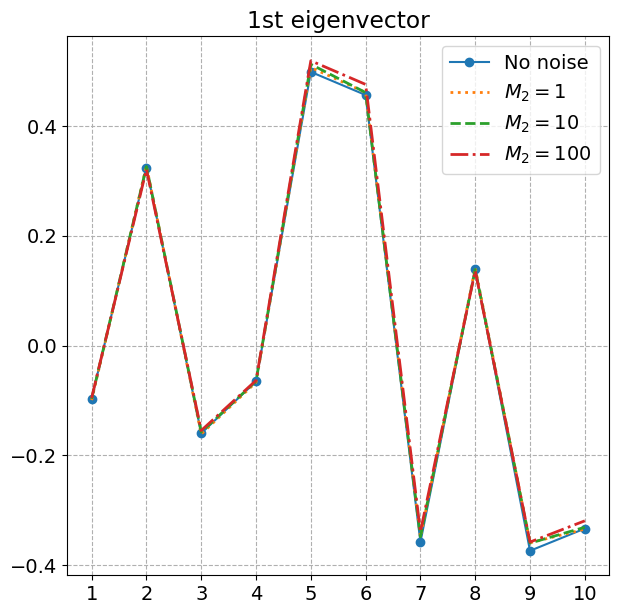}}
    \subfloat[$\sigma=0.1$]{\includegraphics[width=0.3\textwidth]{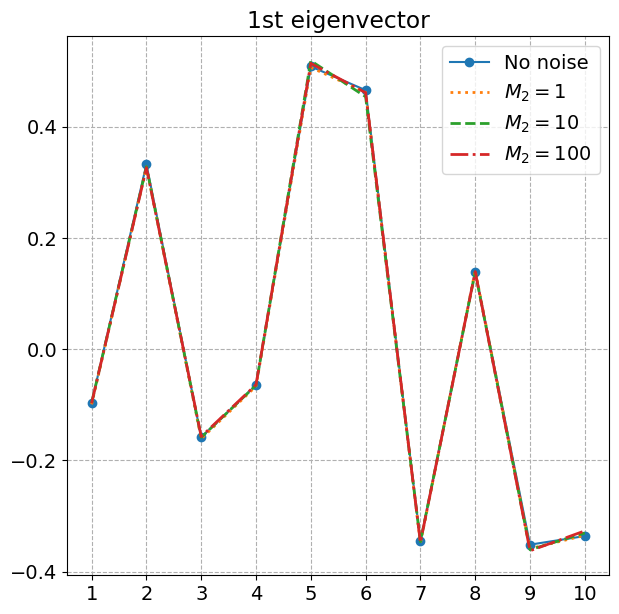}}
    \subfloat[$\sigma=1$]{\includegraphics[width=0.3\textwidth]{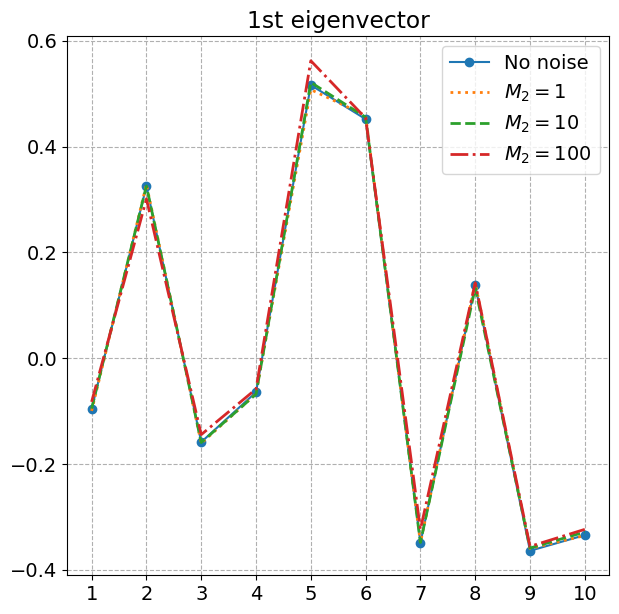}}\\
    \subfloat[$\sigma=0.01$]{\includegraphics[width=0.3\textwidth]{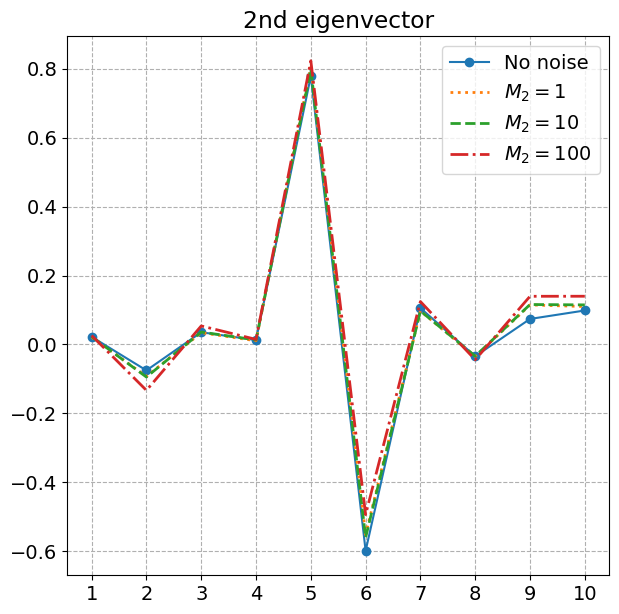}}
    \subfloat[$\sigma=0.1$]{\includegraphics[width=0.3\textwidth]{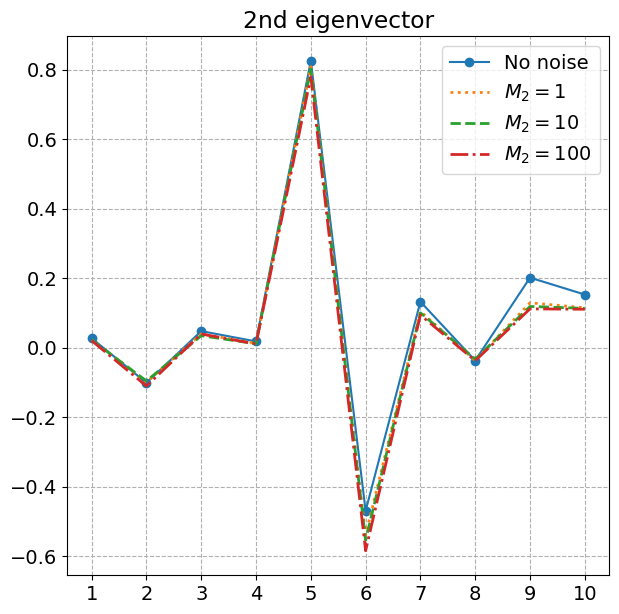}}
    \subfloat[$\sigma=1$]{\includegraphics[width=0.3\textwidth]{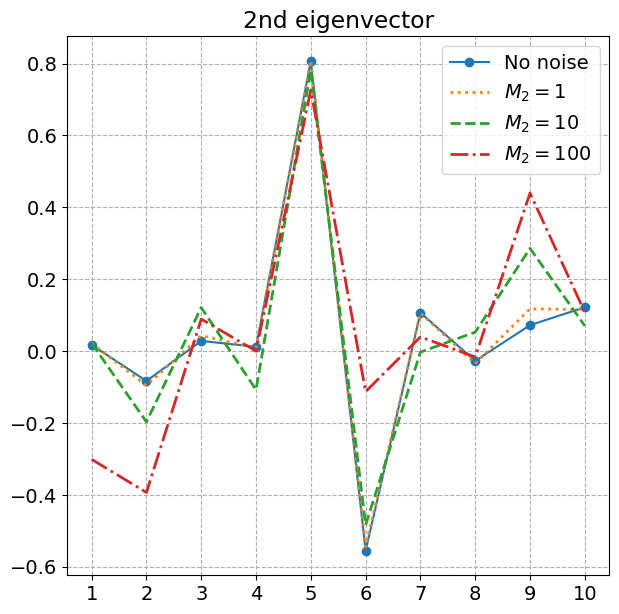}}\\
    \caption{Estimating $\Gamma_i$ and eigenvectors for Example 1: GAS results.}
    \label{testgas}
\end{figure}

In this example we compared the two methods in terms of their accuracy in estimating the eigenvalues, $\Gamma_i$ values, and eigenvectors. In the next example we will compare their accuracy when they are used as surrogate models in a computational problem.  

\subsection{Example 2: Stochastic Computer Models}\label{scc}

Stochastic computer models are models where repeated evaluations with the same inputs give different outputs. For the uncertainty quantification and global sensitivity analysis of stochastic computer models, see Fort et al. \cite{fort2021global}, Hart et al. \cite{hart2017efficient}, and Nanty et al. \cite{nanty2016sampling}.

In this example we consider a stochastic computer model of the form $Y=f(\mathbf{W},\mathbf{V})$. Here $\mathbf{W}$ represents the known uncertainty, and $\mathbf{V}$ represents the model uncertainty. The example is the pricing of an arithmetic Asian call option with discounted payoff
\begin{equation}
    g(S_{t_1},\ldots,S_{t_{d}})=e^{-rT}\max\left(\frac1 d\sum_{i=1}^dS_{t_i}-K,0\right),
\end{equation}
where $S_{t}$ is the price of the underlying at time $t$, $T=t_d$ is the expiry, $r$ is the risk-free interest rate, and $K$ is the strike price. We assume the standard Black-Scholes-Merton model for the underlying
\[
dS_t = rS_tdt + \sqrt{V_t}S_tdW^1_t,
\]
where $W^1_t$ is a standard Brownian motion and $V_t$ is the volatility at time $t$. We treat $V_t$ as the model uncertainty, represented as a function $H$ of $t,V_{<t},S_{<t} $ and $Y_{\leq t}$:
\[
V_t=H(t,V_{<t},S_{<t},Y_{\leq t}),
\]
where $Y_t$ is the unknown random process, $V_{<t}$ and $S_{<t}$ are the complete paths of $V_t$ and $S_t$ before time $t$, and $Y_{\leq t}$ is the complete path of $Y_t$ before and including time $t$.

There are several volatility models in the literature, for example, Bollerslev \cite{bollerslev1986generalized}, Hull and White \cite{hull1987pricing}, Heston \cite{heston1993closed}, Slim \cite{slim2004forecasting}, and Luo et al. \cite{luo2018neural}. We assume the true nature of the process $Y_t$ is unknown (model uncertainty), but in the numerical results, we will use the Heston model to simulate the volatility process:
\begin{equation}
\label{eq_Heston}
\begin{split}
dV_t &= \kappa(\theta-V_t)dt + \sigma\sqrt{V_t}dW^2_t,\\
dW^1_tdW^2_t &= \rho dt,
\end{split}
\end{equation}
where $W^2_t$ is a standard Brownian motion correlated with $W^1_t$. In other words, $W^2_t$ is a substitute for the unknown process $Y_t$ used in the numerical results. (We assume that the Feller condition $2\kappa\theta>\sigma^2$ is satisfied to ensure that $V_t$ is strictly positive - see Albrecher et al. \cite{albrecher2007little}).

Under the risk neutral measure, the option price at time zero is given by:
\begin{equation}
    \E[g(S_{t_1},\ldots,S_{t_{d}})]=\E[f(W^1_{t_1},W^1_{t_2}-W^1_{t_1},\ldots,W^1_{t_{d}}-W^1_{t_{d-1}})],
\end{equation}
where 
\begin{equation}\label{estheston}
f(z_1,z_2,\ldots,z_d)=\E[g(S_{t_1},\ldots,S_{t_{d}})|z_1,z_2,\ldots,z_d],
\end{equation}
and the conditional expectation is over the probability measure that describes the model uncertainty, the process $Y_t$. Note that many payoff functions have finite second-order partial derivatives almost everywhere. For a given path $W^1_{t_1}=z_1,\ldots, W^1_{t_{d}}-W^1_{t_{d-1}}=z_d$, the conditional expectation is estimated, in the numerical results, by simulating ten paths for the volatility process $V_t$ using the Heston model. This approach of conditioning on the distribution for model uncertainty to turn the stochastic computer code to a deterministic one is used in the literature, for example, by Mazo \cite{mazo2019optimal}, and Iooss and Ribatet \cite{iooss2009global}.

We apply the AS and GAS methods to the function $f(z_1,z_2,\dots,z_d)$. Each method yields a different vector $\hat{\pmb w}_1$ of length $d_1$, the vector that contains the important directions of the function. We then construct surrogate models that are functions of $\hat{\pmb w}_1$ for each method. The option price is estimated by computing the expected value of the surrogate model corresponding to AS and GAS methods.

In our numerical results, we choose $T=1,d=10,S_0=K=100, V_0=0.04, r=0.03, \kappa=2, \theta=0.04, \sigma=0.09$, and $\rho = 0.9$. For the AS method, we use a Monte Carlo sample size of $10,000$ to estimate the expectations in the matrix $C$. For the GAS method, we set $M_1=10,000$ and $M_2=1$ (see Eqn. (\ref{gMC_C})). When estimating $\Gamma_i$'s, we set $M_1=10,000$ and $M_2=10$.

Table \ref{cv_2} displays the normalized estimated eigenvalues $\hat\lambda_i/\sum_{k=1}^d\hat\lambda_k$ for the AS method, and normalized estimated values for $\hat\Gamma_i$'s, that is, $\hat\Gamma_i/ \sum_{k=1}^d\hat\Gamma_k$, for $i=1,\ldots,10$. The gap between the first and second directions are the largest in both methods, and thus we choose $d_1=1$ for both.


Figure \ref{fcv_2} plots the coefficients of the first eigenvector of $\hat{\pmb C}$ for each method. We observe that their first eigenvectors are quite different, which means the resulting subspaces will be different.

\begin{table}[htpb!] 
    \centering
    \small
    \begin{tabular}{ccccccccccc}
         & 1 & 2 & 3 & 4 & 5 & 6 & 7 & 8 & 9 & 10\\
        \hline
        AS & 0.556 & 0.056 & 0.054 & 0.051 & 0.050 & 0.049 & 0.047 & 0.047 & 0.045 & 0.045\\
        GAS & 0.981 & 0.003 & 0.003 & 0.003 & 0.002 & 0.002 &0.002 & 0.002 & 0.001 & 0.001
    \end{tabular}
    \caption{Normalized eigenvalues of AS method, and normalized $\Gamma_i$'s of GAS method for Example 2}
    \label{cv_2} 
\end{table}

\begin{figure}[htpb!] 
    \centering 
\includegraphics[scale=0.5]{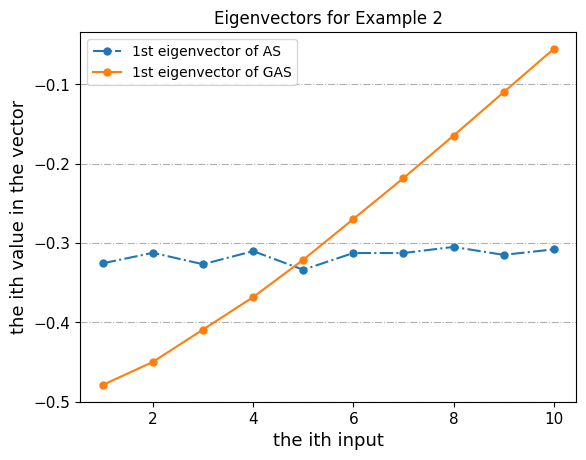} 
    \caption{Coefficients of the first eigenvector of $\hat{\pmb C}$ for Example 2} 
    \label{fcv_2} 
\end{figure}

We choose $N=10,000$ and $N_1=1$ to construct the surrogate models in Algorithm \ref{alg:gas}. To estimate the integral $\E[f(\pmb z)], \pmb z=(z_1,\ldots,z_d)$, we consider the following four estimators:

\begin{itemize}

\item Monte Carlo (MC): $\frac1N\sum_{i=1}^Nf(\pmb z^{(i)})$, $\pmb z^{(i)}\sim \mathcal{N}(\pmb 0,\pmb{I_{d}}), i=1,\ldots,N$. 

\item Polynomial chaos expansion (PCE): $\hat k_0$, where $\hat k_0$ is obtained from the estimated PCE constructed from the data $(\pmb z^{(i)},f(\pmb{z}^{(i)})),\pmb z^{(i)}\sim \mathcal{N}(\pmb 0,\pmb{I_{d}}), i=1,\ldots,N$ (see Eqns. (\ref{eq_est_pce}) and (\ref{LS})).

\item AS\_PCE: $\hat k_0$, where $\hat k_0$ is obtained from the estimated PCE constructed from the data $(\pmb z^{(i)},\hat G_{mc}(\pmb{z}^{(i)})), \pmb z^{(i)}\sim \mathcal{N}(\pmb 0,\pmb{I_{d_1}}), i=1,\ldots,N$. Here $\hat G_{mc}(\pmb{z}^{(i)})$ is computed as in Constantine \cite{constantine2015active}.

\item GAS\_PCE: $\hat k_0$, where $\hat k_0$ is obtained from the estimated PCE constructed from the data $(\pmb z^{(i)},\hat G_{mc}(\pmb{z}^{(i)})), \pmb z^{(i)}\sim \mathcal{N}(\pmb 0,\pmb{I_{d_1}}), i=1,\ldots,N$. Here $\hat G_{mc}(\pmb{z}^{(i)})$ is computed using Algorithm \ref{alg:gas}.

\end{itemize}

For each estimator, we repeat the numerical approximation $40$ times, independently, and obtain estimates $E_1,\ldots,E_{40}$. We then compute the mean square error (MSE) of the estimates:
\begin{equation}
MSE=\frac{1}{40}\sum_{i=1}^{40}(E_i-E_{true})^2,
\end{equation}
where $E_{true}$ is the result of the Monte Carlo estimator using a sample size of $N=10^6$.

We define the efficiency of each method as $1/(Time\times MSE)$, where time is the computing time (seconds) to generate the 40 estimates (including the computation of the matrices $\hat{\pmb C}$, their singular value decomposition, and construction of the surrogate models). Higher efficiency means a better method. Figure \ref{ef2} plots the logarithm of MSE and efficiency of the methods, as a function of the truncation level $p$ for the PCE based methods. The MSE of GAS\_PCE and PCE are similar, and they are significantly smaller than that of AS\_PCE. In terms of efficiency, the best method as the truncation level increases is GAS\_PCE.

\begin{figure}[htpb!] 
    \centering
    \includegraphics[scale=0.4]{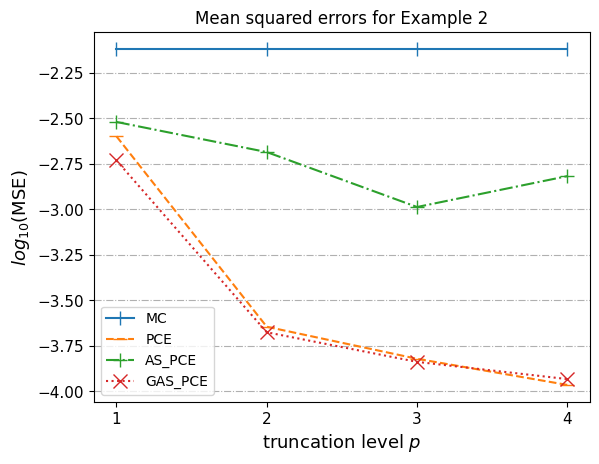}
    \includegraphics[scale=0.4]{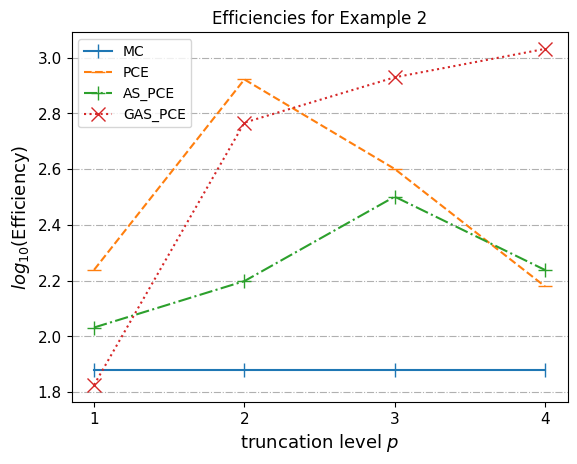}
    \caption{MSEs and efficiencies for Example 2}
    \label{ef2}
\end{figure}

Figure \ref{ef2} compares the methods for a particular choice of parameters $\sigma$ and $\rho$ (see Eqn. (\ref{eq_Heston})) that determine the randomness due to model uncertainty. To compare the methods for a variety of choices for $\sigma$ and $\rho$, we proceed as follows. We select several values for $\sigma$ and $\rho$ while keeping the other parameters as constants. We choose $T=1,d=10,S_0=K=100, V_0=0.025, r=0.03, \kappa=3, \theta=0.025, \sigma= 0.01, 0.05, 0.1, 0.15, 0.2, \rho = -0.99, -0.9, -0.5, 0, 0.5, 0.9, 0.99$. These parameter values are based on Table 5 of Escobar and Gschnaidtner \cite{escobar2016parameters}. For each set of parameters, we compute the PCE using a truncation level of $p=3$ for all the PCE based methods (PCE, AS\_PCE, GAS\_PCE). We then calculate the ratio of the MSE of GAS\_PCE to the MSE of AS\_PCE, and the ratio of the efficiency of GAS\_PCE to the efficiency of AS\_PCE. 
  
Figures \ref{stoc1}-\ref{stoc100} are heatmaps displaying the ratios of MSE and efficiency for several choices for $M_1,M_2$ for the GAS method: $M_1=10,000, M_2=1$; $M_1=1,000, M_2=10$; and $M_1=100, M_2=100$. For the AS method the Monte Carlo sample size is $10,000$. In the figures the MSE and efficiency ratios that are less than one are displayed in blue, and ratios that are larger than one in red. Therefore blue indicates the GAS method performs better than AS for the corresponding parameter choices. 

When $M_2$ is 1, 10 or 100, the GAS method consistently has lower MSE, and better efficiency, for every choice of parameters, and the improvements are quite significant for some choices. 

\begin{figure}[htpb!] 
    \centering
    \includegraphics[scale=0.45]{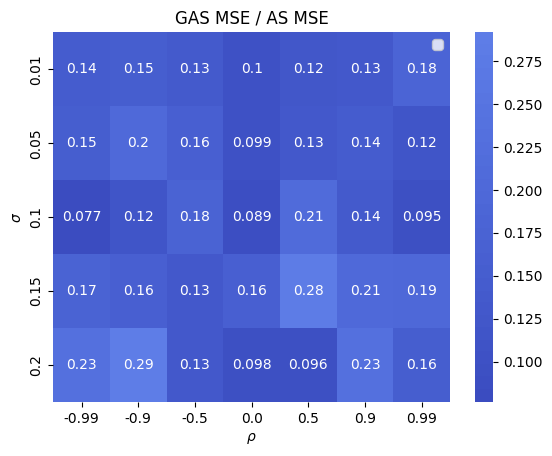}
    \includegraphics[scale=0.45]{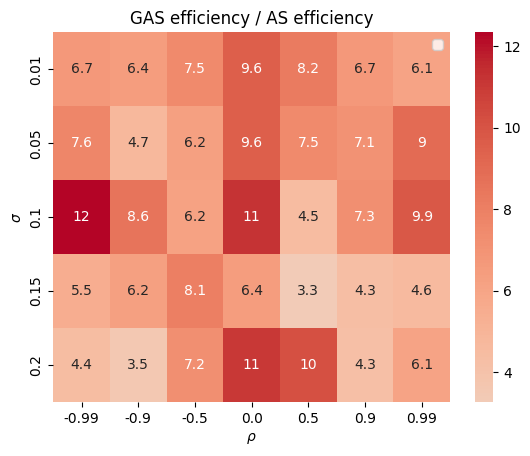}
    \caption{MSE and efficiency ratios for $M_2=1$}
`   \label{stoc1} 
\end{figure}

\begin{figure}[htpb!] 
    \centering
    \includegraphics[scale=0.45]{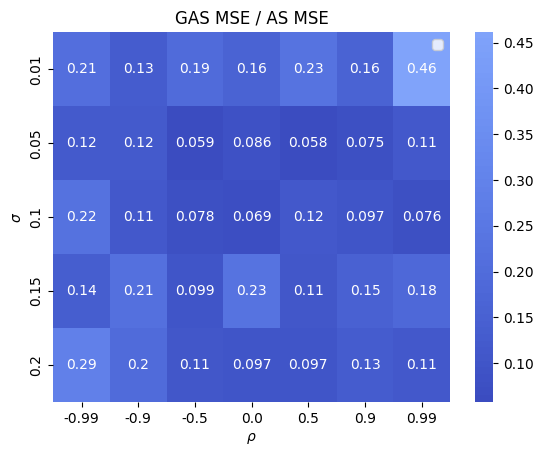}
    \includegraphics[scale=0.45]{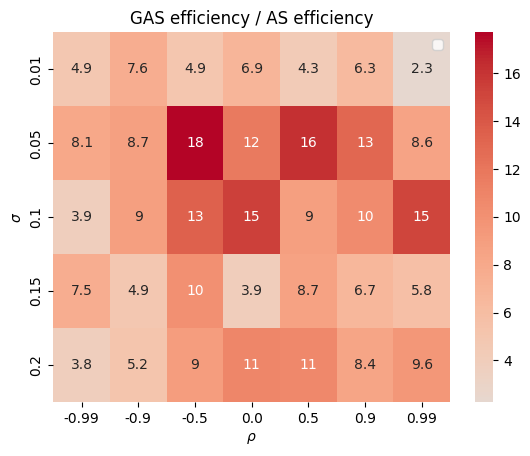}
    \caption{MSE and efficiency ratios for $M_2=10$}
`   \label{stoc10} 
\end{figure}

\begin{figure}[htpb!] 
    \centering
    \includegraphics[scale=0.45]{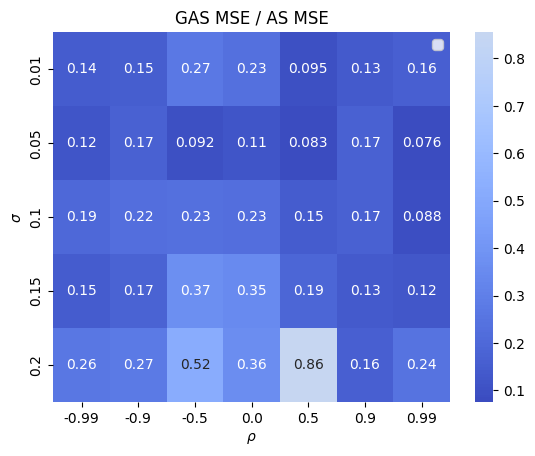}
    \includegraphics[scale=0.45]{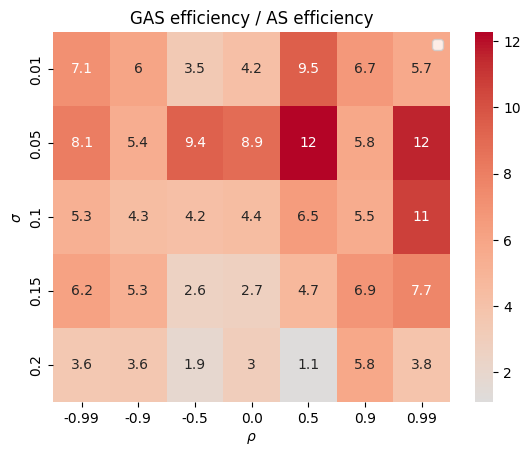}
    \caption{MSE and efficiency ratios for $M_2=100$}
`   \label{stoc100} 
\end{figure}

\subsection{Example 3: Discontinuous Functions}

The AS method requires the existence of the partial derivatives of the function: the entries of the matrix $\pmb C$ are expected values of partial derivatives. The GAS method does not require the existence of partial derivatives, since the entries of the matrix $\pmb C$ for GAS are expected values of finite-differences. However, the error analysis of the GAS method as well as AS method assumes the existence of partial derivatives. In this example we apply GAS to a discontinuous function to further examine the robustness of the method when assumptions used in its theoretical analysis are violated.

Hoyt and Owen \cite{hoyt2020mean} consider the mean dimension of ridge functions, including the cases when the ridge functions are discontinuous. We consider the following example from their paper
\begin{equation}
f(\pmb z) = \pmb 1\{\theta^T\pmb z>0\}, \pmb z\in\mathbb R^d,
\end{equation}
where $\pmb 1 \{\theta^T\pmb z>0\}$ equals 1 if $\theta^T\pmb z>0$ and zero otherwise, and we assume the domain of $f$ is equipped with the CDF of a $d$-dimensional standard normal random variable. 

We consider two cases: $d=10$ and $d=20$. In each case, we choose $\theta$ by randomly generating a $d$-dimensional vector from the multivariate standard normal distribution. We then estimate $\Gamma_i$'s and the first eigenvector of $\hat{\pmb C}$ for GAS method. For the GAS method, we take $M_2 = 1, 10, 100$ while maintaining $M_1M_2=10,000$. Figure \ref{disc_fig} plots the first estimated eigenvector for each method, together with the vector $\theta$, for $d=10$ and $d=20$. In all cases, the estimated eigenvectors are able to capture $\theta$; the direction of the jump in the function.



\begin{figure}[htpb!] 
    \centering 
    \subfloat[$d=10$]{\includegraphics[width=0.45\textwidth]{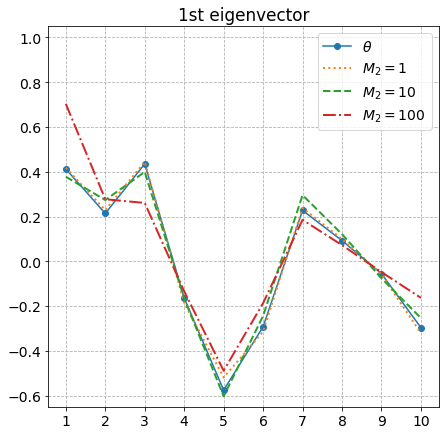}}
    \subfloat[$d=20$]{\includegraphics[width=0.45\textwidth]{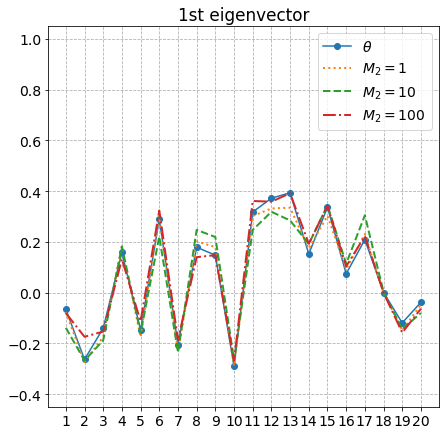}}\\
    \caption{Coefficients of the first eigenvector of $\hat{\pmb C}$ for GAS method, Example 3} 
    \label{disc_fig} 
\end{figure}

\subsection{Example 4: An Ebola Model}
In this example we consider the global sensitivity analysis of a modified SEIR model for the spread of Ebola introduced in Diaz et al. \cite{diaz2018modified}. The model is a system of seven differential equations which describe the dynamics of disease within a population. The output we are interested in is the basic reproduction number $R_0$, and Diaz et al. \cite{diaz2018modified} shows that $R_0$ is
\begin{equation}
R_0 = \frac{\beta_1 + \frac{\beta_2\rho_1\gamma_1}{\omega} + \frac{\beta_3}{\gamma_2}\psi}{\gamma_1 + \psi},
\end{equation}
where $\beta_1, \beta_2, \beta_3, \rho_1, \gamma_1, \gamma_2, \omega$ and $\psi$ are the model parameters. Table \ref{ebola_para}, which is from \cite{diaz2018modified}, displays the distributions of the parameters, obtained using data from the country Liberia.

\begin{table}[htpb!] 
    \centering
    \begin{tabular}{ccccc}
        Parameter & $\beta_1$ & $\beta_2$ & $\beta_3$ & $\rho_1$\\
        \hline
        Liberia & $U(.1, .4)$ & $U(.1, .4)$ & $U(.05, .2)$ & $U(.41, 1)$\\
        

        \hline
        Parameter & $\gamma_1$ & $\gamma_2$ & $\omega$ & $\psi$\\
        \hline
        Liberia & $U(.0276, .1702)$ & $U(.081, .21)$ & $U(.25, .5)$ & $U(.0833, .7)$
    \end{tabular}
    \caption{Parameter distributions in Ebola model}
    \label{ebola_para} 
\end{table}

The problem can be rewritten as $f(\pmb z)=R_0$, where $z_i\sim U(0,1),i=1,\dots,8$. We apply the AS and GAS methods to $f$ using a Monte Carlo sample size of $10,000$ and $h=0.001$ for AS, and $M_1=1000$ and $M_2=10$ for GAS. In this example, for the GAS method, we consider the eigenvalues and eigenvectors of the corresponding $\pmb C$ (as opposed to estimating $\Gamma_i's$) and compare them directly with the eigenvalues and eigenvectors of the $\pmb C$ for the AS method. Tables \ref{lib_table} displays the normalized eigenvalues and the first eigenvector of $\hat{\pmb C}$ for AS and GAS methods.

\begin{table}[htpb!] 
    \centering
    \footnotesize
    \begin{tabular}{ccccccccc}
        Normalized eigenvalues & 1 & 2 & 3 & 4 & 5 & 6 & 7 & 8\\
        \hline
        AS & 0.769 & 0.198 & 0.018 & 0.010 & 0.004 &  0.  &  0.  &  0.\\
        GAS & 0.797 & 0.132 & 0.040 & 0.019 & 0.008 & 0.003   &  0.001  &  0.\\
        \hline
        $1$st eigenvector & 1 & 2 & 3 & 4 & 5 & 6 & 7 & 8\\
        \hline
        AS & 0.385 & 0.062 & 0.340 & 0.043 & -0.252 & -0.298 & -0.038 & -0.759 \\
        GAS & 0.464 & 0.077 & 0.490 & 0.055 & -0.251 & -0.389 & -0.046& -0.565
    \end{tabular}
    \caption{Comparison of results for Liberia}
    \label{lib_table} 
\end{table}


Figures \ref{ebola1AS} and \ref{ebola1} plot the eigenvalues, eigenvectors, and sufficient summary plots for the first active variable $\hat{\pmb u}_1^T\pmb z$, and the first and second active variable $\hat{\pmb u}_2^T\pmb z$ together, for the AS and GAS methods, using the data from Liberia only (data from Sierra Leone give similar conclusions). Here $\hat{\pmb u}_1$ and $\hat{\pmb u}_2$ are the first two eigenvectors, that is the first two columns of $\hat{\pmb{U}}$. The sufficient summary plot for the first active variable is obtained by plotting the output $R_0$ for a sample of $2,000$ randomly generated inputs $\hat{\pmb u}_1^T\pmb z^{(i)},i=1,\ldots,2000.$ The sufficient summary plot for the first and second active variables is also based on 2,000 random samples for each variable. We observe that the first eigenvector for AS and GAS methods are very similar. Their sufficient summary plots show that the output and the first active variable have a similar quadratic relationship, and so as the output and the first two active variables. \footnote{The domain of the function for the AS method as implemented in \cite{diaz2018modified} and the GitHub page by Paul Constantine \url{https://github.com/paulcon/as-data-sets/blob/master/Ebola/Ebola.ipynb} is $(-1,1)^d$. In our results the domain for AS and GAS is $(0,1)^d$. Therefore we cannot compare our AS results with those reported by \cite{diaz2018modified} and Paul Constantine directly.} We conclude both methods give very similar results in this example, while the computational cost on function evaluations and number of operations on matrices are also similar. 

\begin{figure}[htpb!] 
    \centering
    \subfloat[Eigenvalues]{\includegraphics[width=0.4\textwidth]{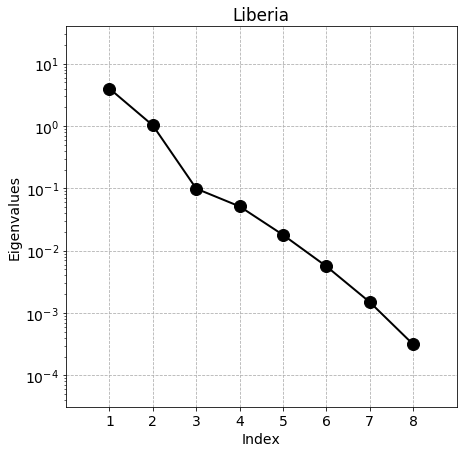}}
    \subfloat[Sufficient summary plot 1]{\includegraphics[width=0.4\textwidth]{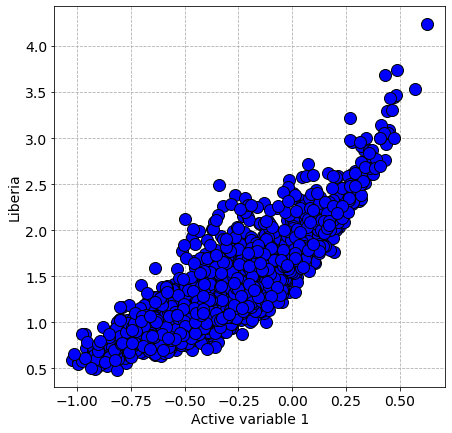}}\\
    \subfloat[Sufficient summary plot 2]{\includegraphics[width=0.4\textwidth]{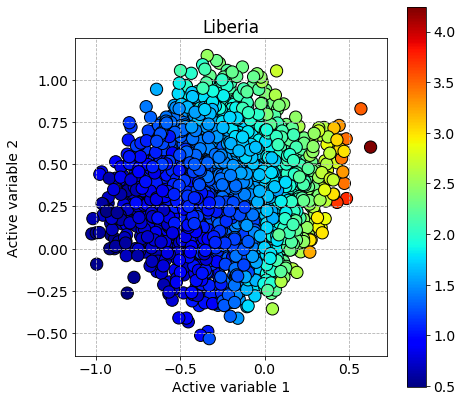}}
    \subfloat[1st eigenvector]{\includegraphics[width=0.4\textwidth]{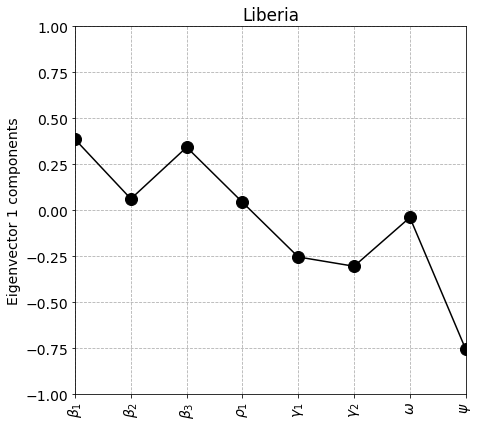}}\\
    \caption{AS results for Ebola model (Liberia)}
    \label{ebola1AS}
\end{figure}

\begin{figure}[htpb!] 
    \centering
    \subfloat[Eigenvalues]{\includegraphics[width=0.4\textwidth]{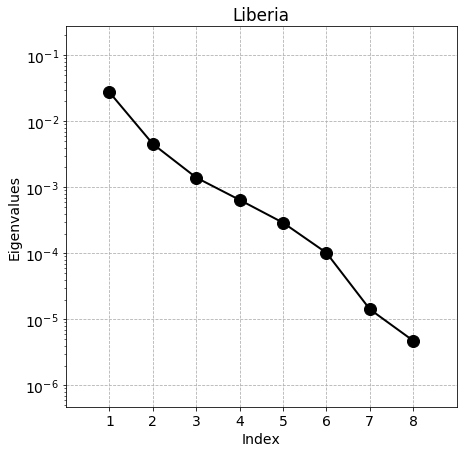}}
    \subfloat[Sufficient summary plot 1]{\includegraphics[width=0.4\textwidth]{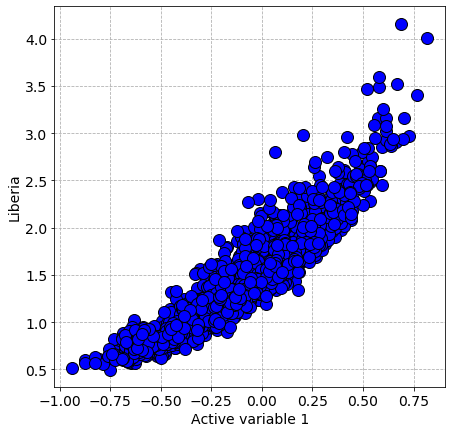}}\\
    \subfloat[Sufficient summary plot 2]{\includegraphics[width=0.4\textwidth]{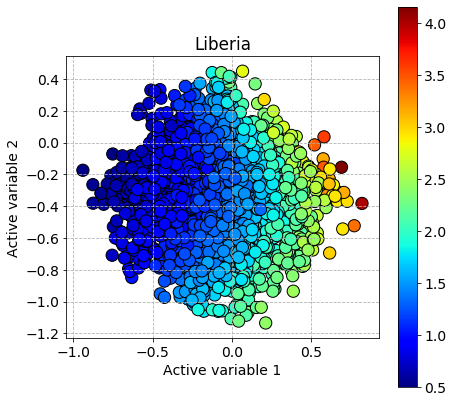}}
    \subfloat[1st eigenvector]{\includegraphics[width=0.4\textwidth]{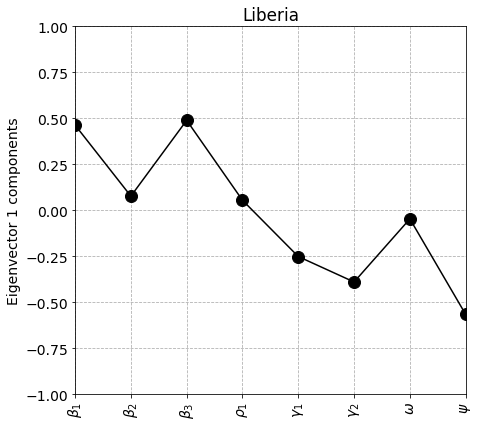}}\\
    \caption{GAS results for Ebola model (Liberia)}
    \label{ebola1}
\end{figure}

\section{Conclusions}
\label{sec:conc}

The global active subspace method uses a global measure for change in function values, the expected value of first-order finite differences, in finding the important directions along which the function changes the most. We have established the error analysis of the method and presented numerical results comparing the method with the active subspace method. Like the active subspace method, the effectiveness of the global active subspace method depends on how fast the $\Gamma_i$'s decay. In the numerical results we have observed that the method is superior to the active subspace method when it is difficult to estimate the gradients, such as in the presence of noise. The method was effective when applied to a non-differentiable function, even though its error analysis requires differentiability. We also observed that the method gave comparable results to the active subspace method when the estimation of the gradient was not problematic, like in the Ebola model. We believe the global active subspace method is a promising dimension reduction method for complex models.

\section*{Acknowledgments}
We thank Richard Oberlin for helpful discussions. 

\bibliographystyle{siamplain}
\bibliography{references}
\end{document}